\documentclass[12pt]{amsart}
\usepackage{amssymb, enumerate, mdwlist}

\evensidemargin\paperwidth
\advance\evensidemargin-\textwidth
\advance\oddsidemargin-1in
\evensidemargin\oddsidemargin

\topmargin\paperheight
\advance\topmargin-\textheight
\advance\topmargin-1in

\theoremstyle{plain}
\newtheorem{mtheorem}{Theorem}

\theoremstyle{plain}
\newtheorem{theorem}{Theorem}[section]
\theoremstyle{plain}
\newtheorem{proposition}[theorem]{Proposition}
\theoremstyle{plain}
\newtheorem{lemma}[theorem]{Lemma}
\theoremstyle{plain}
\newtheorem{claim}[theorem]{Claim}
\theoremstyle{plain}
\newtheorem{corollary}[theorem]{Corollary}
\theoremstyle{plain}
\newtheorem{assumption}[theorem]{Assumption}
\theoremstyle{plain}
\newtheorem{problem}[theorem]{Problem}
\theoremstyle{plain}
\newtheorem{conjecture}[theorem]{Conjecture}
\theoremstyle{definition}
\newtheorem{definition}[theorem]{Definition}
\theoremstyle{remark}
\newtheorem{remark}[theorem]{Remark}
\theoremstyle{remark}
\newtheorem{example}[theorem]{Example}
\theoremstyle{remark}

\title{Superintrinsic synthesis in fixed point properties}
\author{Masato Mimura}
\address{Mathematical Institute, Tohoku University, 6-3, Aramaki Aza-Aoba, Aoba-ku, Sendai 980-8578, Japan/ \newline EPF Lausanne, SB--MATHGEOM--EGG, Station 8 CH-1015 Lausanne Switzerland}
\email{mimura-mas@m.tohoku.ac.jp\,/\,masato.mimura@epfl.ch}
\date{\today}

\begin{document}

\begin{abstract}
The following natural question arises from Shalom's innovational work (1999, Publ.\ IH\'{E}S): ``Can we establish an \textit{intrinsic} criterion to synthesize relative fixed point properties into the whole fixed point property \textit{without assuming Bounded Generation?}'' This paper resolves this question in the affirmative. Our criterion works for ones with respect to certain classes of Busemann NPC spaces. It, moreover, suggests a further step toward constructing super-expanders from finite simple groups of Lie type.
\end{abstract}

\keywords{fixed point property; bounded generation; Busemann NPC spaces; super-expanders; Kazhdan's property $(\mathrm{T})$}

\maketitle

\section{Introduction}
\noindent
\textbf{Notation and Conventions.} \textit{Metric spaces are always assumed to be complete.} Throughout this paper, $G$ means  a \textit{finitely generated} group, and $\mathcal{X}$ means  a (non-empty) class of (non-empty) metric spaces $X$.  \textit{All group actions on metric spaces are assumed to be isometric}, unless otherwise stated. A geodesic segment/line always means minimal one (that means, an isometric embedding of a real segment/line). For $n\in \mathbb{Z}_{\geq 1}$, $[n]$ means the set $\{1,2,\ldots ,n\}$, and $\mathrm{Sym}(n)$ denotes the  symmetric group on $[n]$. The symbol $R$ means a unital and associative ring, possibly non-commutative. For an action $\alpha\colon G\curvearrowright X$ and $M\leqslant G$, $X^{\alpha(M)}:=\{x\in X: \textrm{for all $h\in M$, $\alpha(h)\cdot x=x$}\}$. Isomorphisms between Banach spaces always mean linear ones. Our commutator convention is $[\gamma_1,\gamma_2]:=\gamma_1\gamma_2\gamma_1^{-1}\gamma_2^{-1}$. The symbol $\mathrm{dist}(\cdot ,\cdot)$ denotes the (ordinary) distance between two subsets of a metric space.

\subsection{Introduction: what is ``superintrinsic synthesis''?}\label{subsection=ontro}
The main object in this paper is \textit{fixed point property} with respect to (\textit{isometric}) actions.

\begin{definition}[Fixed point property]\label{definition=fpp}
For a countable group $\Lambda$ and $\Lambda\geqslant M$, the pair $\Lambda\geqslant M$ is said to have \textit{relative property} $(\mathrm{F}_{\mathcal{X}})$ if for each $X\in \mathcal{X}$ and for all $\alpha\colon \Lambda \curvearrowright X$, $X^{\alpha(M)}\neq \emptyset$. We say that $G$ has \textit{property} $(\mathrm{F}_{\mathcal{X}})$ if $G\geqslant G$ has relative property $(\mathrm{F}_{\mathcal{X}})$.
\end{definition}
Recall that $G$ is assumed to be \textit{finitely generated}. For relative properties, we only require that groups are countable.  Compare with Definition~\ref{definition=relbanach}.

This property has been paid strong attention from various backgrounds, including

\begin{itemize}
   \item if $\mathcal{X}=\mathcal{H}\mathrm{ilbert}=\mathcal{B}_{L_2}$, the class of all Hilbert spaces (equivalently, that of all $L_2$-spaces), then \textit{property $(\mathrm{F}_{\mathcal{H}\mathrm{ilbert}})$ is equivalent to the celebrated Kazhdan's property $(\mathrm{T})$} (the Delorme--Guichardet theorem);
   \item rigidity on group actions on manifolds and geometric superrigidity (the Navas theorem, Pansu, \cite{GKM});
   \item analytic obstruction to (relative) Gromov-hyperbolicity (Pansu, \cite{BP}, \cite{gerasimov}. \cite{puls}, \cite{AL}, and others);
    \item expander graphs and its strengthening (Margulis and \cite{lafforgue}); and 
    \item robust Banach property $(\mathrm{T})$ (\cite{oppenheim}, \cite{delasalle}) and obstruction to approaches to variants of (coarse) Baum--Connes conjectures (\cite{lafforgue}, \cite{WY}, \cite{GLS}).
\end{itemize}
See a book \cite{BHV} for property $(\mathrm{T})$; and \cite{BFGM} and  a survey \cite{nowak} for Banach spaces case. We will see the last three  items in details in Subsection~\ref{subsection=motivations}.

The main goal of the present  paper is to establish a \textit{superintrinsic synthesis} in property $(\mathrm{F}_{\mathcal{X}})$. In what follows, we describe the meanings of \textit{``synthesis"}, being \textit{``intrinsic''}, and being \textit{``superintrinsic''}.

One major method to show property $(\mathrm{F}_{\mathcal{X}})$ is, so to speak, \textit{``the Part and the Whole''} strategy, which dates back to Kazhdan's original work. It consists of
\begin{itemize}
  \item(``Part Step'':) show relative property $(\mathrm{F}_{\mathcal{X}})$ for $G\geqslant M_i$ for $1\leq i\leq l$, where $M_i$'s are ``tractable'' subgroups; and
  \item(``Synthesis Step'':) synthesize relative property $(\mathrm{F}_{\mathcal{X}})$ into the whole property $(\mathrm{F}_{\mathcal{X}})$.
\end{itemize}

\textit{``Synthesis''} means the latter step. This is the main topic of the present paper. \textit{The} difficulty here is to \textit{show $\bigcap_{1\leq i\leq l}X^{\alpha(M_i)}\ne \emptyset$}  for $\alpha\colon G\curvearrowright X$, under the condition that $X^{\alpha(M_i)}\ne \emptyset$ for all $i$. 

In 1999, Publ.\ IH\'{E}S \cite{shalom1999}, Shalom made a breakthrough, by switching from fixed point properties to bounded orbit properties. This is well-known as \textit{Shalom's Bounded Generation argument}.

\begin{definition}\label{definition=bg}
A subset $1_G\in U\subseteq G$ is said to \textit{Boundedly Generate} $G$ if there exists $N\in \mathbb{N}$ such that each $\gamma\in G$ may be written as the product of $N$ (possibly overlapping) elements in $U$.
\end{definition}

We do \textit{not} assume, here, that $U$ is of the form $U=\bigcup_{1\leq i\leq l}C_i$ for $C_i$ cyclic.

\begin{theorem}[Shalom's first intrinsic synthesis, \cite{shalom1999}]\label{theorem=shalom}
Assume that $\mathcal{X}$ satisfies that ``property $(\mathrm{F}_{\mathcal{X}})$ is equivalent to boundedness property of all $($equivalently, some$)$ orbits of all group actions on every $X\in \mathcal{X}$''. Let $M_1,\ldots ,M_l\leqslant G$. Assume the following hypothesis is fulfilled:

\noindent
\underline{$\mathrm{Hypothesis}$:} the union $\bigcup_{1\leq i\leq l}M_i$ Boundedly\ Generates $G$.

Then, relative properties $(\mathrm{F}_{\mathcal{X}})$ for $G\geqslant M_i$ for all $i$ imply property $(\mathrm{F}_{\mathcal{X}})$ for $G$.
\end{theorem}
Examples of such $X\in \mathcal{X}$ in the assumption above include reflexive Banach spaces (Ryll-Nardzewski's theorem); $L$-embedded Banach spaces such as (non-commutative) $L_1$-spaces (\cite{BGM}); and  $\mathrm{CAT}$$(0)$ spaces (see a book \cite{BH} for $\mathrm{CAT}$$(0)$ spaces). The proof is immediate once we observe that for all $x\in X$ and for all $\gamma_1,\gamma_2\in G$,
\[
d(\alpha(\gamma_1\gamma_2)\cdot x,x)\leq d(\alpha(\gamma_1\gamma_2)\cdot x,\alpha(\gamma_1)\cdot x)+d(\alpha(\gamma_1)\cdot x,x) =d(\alpha(\gamma_2)\cdot x,x)+d(\alpha(\gamma_1)\cdot x,x).
\]

The key here is that for bounded orbit properties, we do not need to care exact locations of points to consider orbits. The price to pay is, we need \textit{Bounded} Generation rather than ordinary one, to control errors of orbits from being a singleton.

By Theorem~\ref{theorem=shalom}, Shalom provided the first proof of property $(\mathrm{F}_{\mathcal{H}\mathrm{ilbert}})$ for $\mathrm{SL}(n,\mathbb{Z})$ for $n\geq 3$ without employing $\mathrm{SL}(n,\mathbb{R})$. This was done by combining the following two facts. Here, for $R$ and $n\geq 2$, the \textit{elementary group} $\mathrm{E}(n,R)$ is the subgroup of $\mathrm{GL}(n,R)$ generated by \textit{elementary matrices} $\{e_{i,j}(r): i\ne j\in [n] ,r\in R\}$, where $(e_{i,j}(r))_{l,k}=\delta_{l,k}+r \delta_{i,l}\delta_{j,k}$ ($\delta_{\cdot,\cdot}$ denotes the Dirac delta). The commutator relation
\[
[e_{i,j}(r_1),e_{j,k}(r_2)]=e_{i,k}(r_1r_2) \tag{$*$}
\]
for $i\ne j\ne k\ne i$ implies finite generation of $\mathrm{E}(n,R)$ for a finitely generated $R$ and for $n\geq 3$. In some literature such as \cite{EJ}, $\mathrm{E}(n,R)$ is written as $EL(n,R)$.

\begin{theorem}[Kassabov \cite{kassabov} for general cases]\label{theorem=relT}
Let $G=\mathrm{E}(n,R)$, $M=\langle e_{i,n}(r): i\in [n-1],r\in R\rangle =
\left(\begin{array}{cc}
I_{n-1} & R^{n-1} \\
0 & 1 
\end{array}
\right)(\simeq (R^{n-1},+))$, and $L=\langle e_{n,j}(r): j\in [n-1],r\in R\rangle =
\left(\begin{array}{cc}
I_{n-1} & 0 \\
{}^t(R^{n-1}) & 1 
\end{array}
\right)$. Then, for every $n\geq 3$ and for every finitely generated $R$, $G\geqslant M$ and $G\geqslant L$ have relative property $(\mathrm{F}_{\mathcal{H}\mathrm{ilbert}})$.
\end{theorem}
Note that if $R=\mathbb{Z}$, then $\mathrm{E}=\mathrm{SL}$ by Gaussian elimination. 

\begin{theorem}[Carter--Keller \cite{CK}]\label{theorem=CK}
For $G$, $M$, $L$ as in Theorem~$\ref{theorem=relT}$, if $n\geq 3$ $\mathrm{and\ if}$ $R=\mathbb{Z}$, then $M\cup L$ Boundedly Generates $G$.
\end{theorem}

The synthesis in Theorem~\ref{theorem=shalom} is \textit{``intrinsic''}, more precisely, the hypothesis is stated only in term of structures inside the group, and not of \textit{extrinsic} data such as information on group actions (including spectral data). Therefore, as long as relative properties $(\mathrm{F}_{\mathcal{X}})$ for $G\geqslant M$ and $G\geqslant L$ are proved, where $G,M$ and $L$ are as in Theorem~\ref{theorem=CK} and $\mathcal{X}$ as in Theorem~\ref{theorem=shalom}, we may obtain property $(\mathrm{F}_{\mathcal{X}})$ for $G$ without any extra effort. In this sense, \textit{intrinsic synthesis is robust under changing the class $\mathcal{X}$}. However, to the best knowledge of the author, all intrinsic syntheses in previous work (\cite{shalom1999} and \cite{shalom2006}) imposed some form of \textit{Bounded Generation}. From this background, the following natural question arises from \cite{shalom1999}, which \textit{was} open for more than 15 years.

\begin{problem}[Superintrinsic synthesis problem]\label{problem=shalom}
Can we achieve $\mathrm{intrinsic}$ synthesis whose hypotheses are free from Bounded Generation hypotheses?
\end{problem}

In this paper, we call such synthesis \textit{super}intrinsic one. To the best knowledge of the author, this problem might \textit{have been} considered as being hopeless, because the aforementioned argument makes essential use of Bounded Generation to control $G$-orbits. 

The main results of the present paper are the following.
\begin{itemize}
  \item (\underline{Theorems~\ref{mtheorem=main1}, \ref{mtheorem=main1b}, \ref{mtheorem=main2}, and \ref{mtheorem=sr}}) \textit{Resolutions of Problem~$\ref{problem=shalom}$ in the affirmative}. 
  \item (\underline{Corollaries~\ref{corollary=st} and \ref{corollary=expanders}}) Application of  our criteria to elementary/Steinberg groups, and expanders associated with them (see Subsection~\ref{subsection=motivations}).
  \item (\underline{Corollary~\ref{corollary=superexpanders}}) Reduction of Conjecture~\ref{conjecture=superexpanders}.$(2)$ on \textit{super-expanders coming from special linear groups over finite fields} to \textit{relative  property $(\mathrm{T})$} with respect to uniformly convex Banach spaces (see Definitions~\ref{definition=superexpanders} and \ref{definition=relbanach}). 
\end{itemize}
For precise statements, see Subsection~\ref{subsection=main}. We emphasize that our superintrinsic synthesis may play an essential role in study of fixed point property with respect to a class with ``\textit{unbounded wildness}'', such as the class of (non-commutative) $L_q$-spaces for \textit{all} $q\in (1,\infty)$ or that of all uniformly convex Banach spaces. \textit{Extrinsic} synthesis may have difficulty in such cases. See Subsection~\ref{subsection=motivations} for detailed discussions.

Our synthesis also works for a class of certain \textit{non-linear} metric spaces, such as certain \textit{Busemann Non-Positively Curved space}, hereafter we write as a \textit{BNPC space} for short. See Remark~\ref{remark=bnpc} for the definition. For the third item above, inspired by the great achievement in work of Ershov, Jaikin-Zapirain, and Kassabov \cite[Theorem 9.3]{EJK} for ordinary expanders, we, here, conjecture the following. See Subsection~\ref{subsection=motivations} for  details. For the restriction to rank $\geq 3$, see Remark~\ref{remark=ngeq4}.

\begin{conjecture}[\textit{Unbounded rank super-expanders conjecture}]\label{conjecture=superexpanders}
\begin{enumerate}[$(1)$]
  \item For every prime $p$, the sequence $(\mathrm{SL}(n,\mathbb{F}_{p}))_{n\geq 4}$ $\mathrm{can}$ $\mathrm{form}$ super-expanders $($that means, the sequence of the Cayley graphs of them with respect to suitable choices of system of finite generating sets becomes super-expanders$)$.
  \item For every sequence of primes $(p_n)_n$, $(\mathrm{SL}(n,\mathbb{F}_{p_n}))_{n\geq 4}$ can form super-expanders.
  \item The family of all simple group of Lie type and rank at least $3$ has a mother group with property $(\mathrm{F}_{\mathcal{B}_{\mathrm{uc}}})$.
\end{enumerate}
\end{conjecture}

One remark is that $(b)$.$(1)$ of Corollary~\ref{corollary=st}, in particular, provides the following result, which is one of the main results in \cite{EJ}.
\begin{theorem}[Ershov and Jaikin-Zapirain, Theorem~1.1 in \cite{EJ}]\label{theorem=EJ}
For every $n\geq 3$ and for every finitely generated $R$, $\mathrm{E}(n,R)$ has property $(\mathrm{F}_{\mathcal{H}\mathrm{ilbert}})$.
\end{theorem}
We refer the reader to a short  expository article \cite{mimuraT}, in which we focus on presentation of an alternative proof of Theorem~\ref{theorem=EJ} simpler than the one in \cite{EJ}.  Unlike the original proof in \cite{EJ}, our approach, however, does \textit{not} supply any estimate of Kazhdan constants.

Our work is based on \textit{self-improvement argument}, which is inspired by the second intrinsic synthesis by Shalom \cite[4.III]{shalom2006} (he called his argument \textit{algebraization}). 

\subsection{The heart of our argument: self-improvement, Pseudo-Uniqueness, and assumption $(\mathrm{TP})$}\label{subsection=heart}
Our superintrinsic synthesis (Theorem~\ref{mtheorem=main1}) consists of six assumptions on $\mathcal{X}$; and three hypotheses on $G$, including an hypothesis concerning a ``\textrm{Game}'' that is introduced in the present paper. At a first glance, it might look too complicated. However, all of the assumptions/hypotheses have root in the following simple \textit{self-improvement argument}. For this reason, before proceeding in the statement of our Theorem~\ref{mtheorem=main1}, we will describe this \textit{heart} of the arguments. 

Here, we concentrate on one example to roughly see   how to achieve superintrinsic synthesis in property $(\mathrm{F}_{\mathcal{H}\mathrm{ilbert}})$. Our group here is $G=\mathrm{E}(n,R)$ as in Theorem~\ref{theorem=relT}. Let $M\leqslant G$ and $L\leqslant G$ be as in Theorem~\ref{theorem=relT}. Let $\alpha\colon G\curvearrowright \mathcal{H}$ for a Hilbert space $\mathcal{H}$. By relative property $(\mathrm{F}_{\mathcal{H}\mathrm{ilbert}})$, $\mathcal{H}^{\alpha(M)}\ne \emptyset$ and $\mathcal{H}^{\alpha(L)} \ne \emptyset$. We assume the following.

\begin{assumption}[Existence and Uniqueness assumption]\label{assumption=unique}
There $\mathrm{exists}$ a $\mathrm{unique}$ pair $(\xi,\eta)$, where $\xi \in\mathcal{H}^{\alpha(M)}$ and $\eta \in\mathcal{H}^{\alpha(L)}$, that realizes $D:=\mathrm{dist}(\mathcal{H}^{\alpha(M)},\mathcal{H}^{\alpha(L)})$.
\end{assumption}
Then, our \textit{self-improve argument} goes as follows.
\begin{proposition}[\textit{Heart} of our self-improvement argument]\label{proposition=heart}
Let $G,M,L$ and $(\alpha,\mathcal{H})$ be as in the paragraph above Assumption~$\ref{assumption=unique}$. Then, $\mathrm{under}$ $\mathrm{Assumption}$~$\ref{assumption=unique}$, the realizer $(\xi,\eta)$, in fact, satisfies that $\xi \in \mathcal{H}^{\alpha(G)}$ and $\eta \in \mathcal{H}^{\alpha(G)}$. 
\end{proposition}

\begin{proof}
Recall from our notation that $\alpha$ is isometric. The idea is that we enlarge two subgroups $H_1$ and $H_2$ of $G$ stage by stage such that $\xi \in \mathcal{H}^{\alpha(H_1)}$ and $\eta \in \mathcal{H}^{\alpha(H_2)}$. In initial stage, set $H_1=M$ and $H_2=L$. 

\textit{Our first move $($enlargement$)$} uses $P:=\langle e_{i,j}(r):i\ne j\in [n-1],r\in R\rangle =\left(\begin{array}{cc}
\mathrm{E}({n-1},R) & 0 \\
0 & 1 
\end{array}
\right)$. Let $h\in P$. Observe $hH_1h^{-1}(:=hMh^{-1})\geqslant M$ and $hH_2h^{-1}(:=hLh^{-1})\geqslant L$. It exactly says that $\alpha(h)\cdot \xi \in \mathcal{H}^{\alpha(M)}$ and $\alpha(h)\cdot \eta \in \mathcal{H}^{\alpha(L)}$. By isometry of $\alpha$, $(\alpha(h)\cdot \xi ,\alpha(h)\cdot \eta)$ is \textit{another realizer} of $D$. By uniqueness, this must coincide with $(\xi,\eta)$. Therefore, 
\[
\xi \in \mathcal{H}^{\alpha(M)}\cap \mathcal{H}^{\alpha(P)}=\mathcal{H}^{\alpha(\langle M,P\rangle)}.
\]
Thus, if we set \textit{new} $H_1$ and $H_2$ as 
\[
H_1:=\langle M,P\rangle =\left(\begin{array}{cc}
\mathrm{E}({n-1},R) & R^{n-1} \\
0 & 1 
\end{array}
\right), \textrm{ and }H_2:=\langle L,P\rangle =\left(\begin{array}{cc}
\mathrm{E}({n-1},R) & 0 \\
{}^t(R^{n-1}) & 1 
\end{array}
\right),
\]
then $\xi \in \mathcal{H}^{\alpha(H_1)}$ and $\eta\in \mathcal{H}^{\alpha(H_2)}$. This is the first move in our \textit{self-improvement argument}. This specific move is observed by Shalom \cite[4.III]{shalom2006}.

\textit{Our second move} employs $w=\left(\begin{array}{ccc}
0 & 0 & 1\\
0 & I_{n-2} & 0 \\
-1 & 0 & 0
\end{array}
\right)$ (note that $w=$$e_{1,n}(1)e_{n,1}(-1)e_{1,n}(1)$$\in G$). Then, thanks to the first enlargement, a \textit{miracle} happens on these new $H_1$ and $H_2$ as follows:
\begin{align*}
 wH_2w^{-1}=\left(\begin{array}{cc}
1 & {}^t(R^{n-1}) \\
0 & \mathrm{E}({n-1},R) 
\end{array}
\right) \geqslant M,\ \mathrm{and} \quad 
wH_1w^{-1}=\left(\begin{array}{cc}
1 & 0 \\
R^{n-1} &  \mathrm{E}({n-1},R)
\end{array}
\right) \geqslant L.
\end{align*}
\textit{It follows that $\alpha(w)\cdot \eta\in \mathcal{H}^{\alpha(M)}$ and $\alpha(w)\cdot \xi\in \mathcal{H}^{\alpha(L)}$}. Again, by isometry of $\alpha$, this time \textit{$(\alpha(w)\cdot\eta,\alpha(w)\cdot\xi)$ is another realizer}. By uniqueness, $\eta=\alpha(w)\cdot \xi$. Recall that $\eta \in \mathcal{H}^{\alpha(H_2)}$. Hence,
\[
\xi \in \mathcal{H}^{\alpha(H_1)}\cap \mathcal{H}^{\alpha(w^{-1}H_2w)}=\mathcal{H}^{\alpha(\langle H_1,w^{-1}H_2w\rangle)}.
\]
Note that $\langle H_1,w^{-1}H_2w\rangle= \left\langle\left(\begin{array}{cc}
\mathrm{E}({n-1},R) & R^{n-1} \\
0 & 1 
\end{array}
\right),\left(\begin{array}{cc}
1 & {}^t(R^{n-1}) \\
0 & \mathrm{E}(n-1,R) 
\end{array}
\right) \right\rangle$ equals $G$ by $(\ast)$. Therefore, $\xi\in \mathcal{H}^{\alpha(G)}$, as desired. Similarly, $\eta\in \mathcal{H}^{\alpha(G)}$.
\end{proof}
In this way, under Assumption~\ref{assumption=unique}, we overcome \textit{the difficulty in showing $\mathcal{H}^{\alpha(M)}\cap \mathcal{H}^{\alpha(L)}\ne \emptyset$}. This solution does \textit{not} use switching to bounded orbit properties, and thus, we remove Bounded Generation hypotheses.

There are, needless to say, \textit{gaps} in both of \textit{existence} and \textit{uniqueness} in Assumption~\ref{assumption=unique}. One on existence will be fixed by taking (pointed) \textit{metric ultraproducts} with \textit{rescaling}, see Section~\ref{section=up} for certain actions. One on uniqueness may be fixed once we know the existence, in the following \textit{pseudo}-manner. Recall that every Hilbert space $\mathcal{H}$ is \textit{strictly convex}, that means, for all $\xi, \eta \in \mathcal{H}$ with $\|\xi\|=\|\eta\|=1$, $\xi\ne \eta$ implies ($\|\xi+\eta\|/2)<1$. Recall, in addition, that for $\alpha\colon G\curvearrowright \mathcal{H}$, after regarding $\mathcal{H}$ as a real Hilbert space, $\alpha$ turned out to be affine (the Mazur--Ulam theorem. In this case, it follows immediately by strict convexity of $\mathcal{H}$). Hence, $\alpha$ is decomposed into the \textit{linear part} $\pi$ and \textit{cocycle part} $b$, that means, for all $g\in G$ and for all $\zeta\in \mathcal{H}$, $\alpha(g)\cdot \zeta =\pi(g)\zeta+b(g)$. Here $\pi$ is a (real) unitary representation, and $b\colon G\to \mathcal{H}$ satisfies $b(\gamma_1\gamma_2)=b(\gamma_1)+\pi(\gamma_1)b(\gamma_2)$ (\textit{cocycle identity}). 

\begin{lemma}[``Pseudo-Uniqueness of realizers'']\label{lemma=pseudounique}
Let $G,M,L$ and $(\alpha,\mathcal{H})$ be as in the paragraph above Assumption~$\ref{assumption=unique}$. Let $\pi$ be the linear part of $\alpha$. Assume that $(\xi,\eta)$ and $(\xi',\eta')$ are realizers of $D:=\mathrm{dist}(\mathcal{H}^{\alpha(M)},\mathcal{H}^{\alpha(L)})$. Then, $\xi-\xi'=\eta-\eta'\in \mathcal{H}^{\pi(G)}$.
\end{lemma}

\begin{proof}
This is the \textit{parallelogram argument by Shalom} \cite[4.III.6]{shalom2006}. Let $m_1:=(\xi+\xi')/2$ and $m_2:=(\eta+\eta')/2$. Then, we claim that $(m_1,m_2)$ is again a realizer of $D$. Indeed, $m_1\in \mathcal{H}^{\alpha(M)}$ and $m_2\in \mathcal{H}^{\alpha(L)}$ because $\alpha$ is affine. By triangle inequality, $\|m_1-m_2\|\leq (D+D)/2=D$. By minimality of $D$, $\|m_1-m_2\|$, in fact, equals $D$. 

Therefore, the strict convexity of $\mathcal{H}$ implies that $\xi-\eta=\xi'-\eta'$ (otherwise, $\|m_1-m_2\|<D$). Hence, $\xi-\xi'=\eta-\eta'$. The containment in the assertion holds because $\xi-\xi'\in \mathcal{H}^{\pi(M)}$, $\eta-\eta'\in \mathcal{H}^{\pi(L)}$, and $\langle M,L\rangle=G$ (by $(\ast)$).
\end{proof}

In this paper, we observe that this argument may be generalized even to \textit{non-linear} case. Our idea is to \textit{replace ``$y-x$'' with the parallel equivalence class of $[x,y]$}'', see Section~\ref{section=parallel}. Here, we employ the concept of \textit{parallelism} on oriented geodesic segments, which  was studied by Gelander--Karlsson--Margulis \cite{GKM}.

\begin{definition}[Parallelism on geodesic segments; and assumption $(\mathrm{TP})$]\label{definition=parallelism}
Let $(X,d)$ be a uniquely geodesic space. For $x,y\in X$, denote by $[x,y]$ the  (oriented) geodesic segment from $x$ to $y$. For $0\leq t\leq 1$, $m_t[x,y]$ means the point on $[x,y]$ that divides $[x,y]$ internally at ratio $t:(1-t)$.
\begin{enumerate}[$(1)$]
  \item We say $[x,y]\,||\,[x',y']$ if $d(x,x')=d(y,y')=d(m_{1/2}[x,y],m_{1/2}[x',y'])$ holds.
  \item We say that $X$ satisfies the \textit{assumption $(\mathrm{TP})$} (``\textit{transitive law of paralellisms}'') if $[x,y]\,||\,[x',y']$ and $[x',y']\,||\,[x'',y'']$ imply $[x,y]\,||\,[x'',y'']$. We say that $\mathcal{X}$ satisfies $(\mathrm{TP})$ if every $X\in \mathcal{X}$ fulfills $(\mathrm{TP})$.
\end{enumerate}
\end{definition}

\textit{Under $(\mathrm{TP})$, the binary relation ``$||$'' is an equivalence relation}.

\noindent
\textbf{The organization of this paper.} In Section~\ref{section=pre}, we explain history and motivations on synthesis in fixed point properties, and state our main theorems and corollaries. In Section~\ref{section=hypotheses}, we state six assumptions  on $\mathcal{X}$, and present the definition of $(\mathrm{Game})$ and $(\mathrm{Game}^{+})$ in our main theorems. In Section~\ref{section=parallel}, we exhibit the rigorous form of our \textit{self-improvement argument} for type $(\mathrm{I})$ moves (Proposition~\ref{proposition=key}). In Section~\ref{section=up}, we recall well-known arguments to obtain a (uniform) action with a realizer of the distance by employing (pointed) metric ultraproducts. Section~\ref{section=proof} is devoted to the proofs of Theorems~\ref{mtheorem=main1} and \ref{mtheorem=main1b}. In Section~\ref{section=general}, we verify two generalizations of our main theorems: one for more than two subgroups (Theorem~\ref{mtheorem=main2}), and the other for superreflexive Banach spaces (Theorem~\ref{mtheorem=sr}). In Section~\ref{section=application}, we prove some applications of superintrinsic synthesis, more precisely, Corollaries~\ref{corollary=st}, \ref{corollary=expanders}, and \ref{corollary=superexpanders}.

\section{Motivations and main results}\label{section=pre}
\subsection{Motivations: comparison with extrinsic synthesis}\label{subsection=motivations}
We sketch the history and motivations to clarify  why we should study \textit{superintrinsic} synthesis, in comparison to intrinsic ones with Bounded Generation, and to \textit{extrinsic} synthesis. The bottle-neck of previous intrinsic syntheses \cite{shalom1999}, \cite{shalom2006} \textit{was} that Bounded Generation hypotheses are too strong. For instance, it is known that the assertion of Theorem~\ref{theorem=CK} is \textit{false} for $R=\mathbb{C}[t]$ (van der Kallen). For general rings such as $R=\mathbb{Z}[t]$, $\mathbb{F}_q\langle s,t \rangle$ (here $\langle \rangle$ means the non-commutative polynomial ring), the situation is open, and might be expected to be negative. Hence, despite Theorem~\ref{theorem=relT}, intrinsic synthesis \textit{was} unable to synthesize them into property $(\mathrm{F}_{\mathcal{H}\mathrm{ilbert}})$ for $\mathrm{E}(n,R)$, unless $R=\mathbb{Z}$ (\cite{shalom1999}), $R=\mathbb{F}_p[t]$ (\cite{nica}) for prime $p$, and relatives; or $R$ is commutative (\cite{shalom2006} by combination with Vaserstein's Bounded Generation \cite{vaserstein}). 

To attack the case of general non-commutative $R$, Ershov and Jaikin-Zapirain \cite{EJ} switched from intrinsic to \textit{extrinsic} synthesis. Their work is inspired by Dymara--Januszkiewicz \cite{DJ}. To be more precise on the word ``extrinsic'', they considered \textit{angles} between fixed point subspaces by subgroups. They proved that if these angles are \textit{sufficiently close to orthogonal} (they made quantitative comparison with absolute thresholds), then synthesis works. See \cite[Subsection~1.2 and Section~2]{lavy} as well as \cite{EJ} and  \cite{kassabovex} for precise statements. As a byproduct, they established Theorem~\ref{theorem=EJ} with explicit estimates of Kazhdan constants.

This extrinsic synthesis is named \textit{almost orthogonal} argument. It has been developed by Kassabov \cite{kassabovex} and  Ershov--Jaikin-Zapirain--Kassabov \cite{EJK} for property $(\mathrm{T})$; Oppenheim \cite{oppenheim} with respect to a wide class $\mathcal{X}$ of Banach spaces; and Lavy \cite{lavy} with respect to Hadamard manifolds. These extrinsic methods work powerfully for the case that $R$ is an algebra over $\mathbb{F}_p$ for sufficiently large prime $p$, and that $\mathcal{X}$ is of ``\textit{bounded wildness}''. For instance, $L_q$-spaces for $q$ in a \textit{bounded} range, or for Banach spaces of \textit{bounded} range,  away from $1$, of type and cotype  (see a book \cite{BL} for the definitions) in \cite{oppenheim}. 

The reason to develop \textit{intrinsic} synthesis, apart from glory of \textit{extrinsic} ones, is the following. In several cases, there is need to study property $(\mathrm{F}_{\mathcal{X}})$ with respect to $\mathcal{X}$ of ``\textit{unbounded wildness}''. It may be \textit{intrinsic} synthesis that will shed light on this study. Here, we exhibit three examples.

\noindent
\textbf{(1) Analytic obstruction to relative Gromov-hyperbolicity.} For each $q\in [1,\infty)$, let $\mathcal{B}_{L_q}$ be the class of all $L_q$-spaces (the underlying measure spaces vary). More generally, let $\mathcal{B}_{\mathrm{NC}L_q}$ be the class of all non-commutative $L_q$-spaces (the underlying von Neumann algebras vary, see \cite{PX} for the definition). Then, $\mathcal{B}_{L_q}\subseteq \mathcal{B}_{\mathrm{NC}L_q}$. Let $\mathcal{B}_{\mathrm{Q}L_q}$ be the class of all quotient spaces of $L_q$-spaces.
\begin{theorem}[Bourdon--Pajot \cite{BP}; Section~1.5, Corollary in \cite{gerasimov} and Theorem~1.3 in \cite{puls} for relatively hyperbolic groups; and Alvarez--Lafforgue \cite{AL}]\label{theorem=relhyp}

Let $G$ be an infinite group with property $(\mathrm{F}_{\mathcal{B}_{L_q}})$ for $\mathrm{all}$ $q\in [2,\infty)$. Then, $G$ is $\mathrm{never}$ hyperbolic. Moreover, $G$ is $\mathrm{never}$ relatively hyperbolic $($with respect to finite collection of infinite subgroups$)$. 

Let $G$ be an infinite group with property $(\mathrm{F}_{\mathcal{B}_{\mathrm{Q}L_{q_n}}})$ for some $(q_n)_n$ with $\lim_{n\to \infty}q_n=1$. Then, $G$ is $\mathrm{never}$ hyperbolic.
\end{theorem}
For this application, there is a \textit{major distinction} between ``property $(\mathrm{F}_{\mathcal{B}_{L_q}})$ for \textit{bounded} $q$'' and ``that for \textit{all} $q\in [2,\infty)$''. Also, ``property $(\mathrm{F}_{\mathcal{B}_{\mathrm{Q}L_{q_n}}})$ for some $(q_n)_n$ with $\lim_{n\to \infty}q_n=1$'' has special interest.

\noindent
\textbf{(2) Super-expanders.}
\begin{definition}\label{definition=mothergroup}
Let $(G,S)$ be a pair of a group and its finite generating set. This is said to be a \textit{mother group} of a sequence $((G_n,S_n))_{n\in \mathbb{N}}$ if $|G_n|<\infty$ with $\lim_{n\to \infty}|G_n|\to \infty$, and if there exists a sequence of group homomorphisms $(\phi_n)_n$ such that $\phi_n$ maps $G$ onto $G_n$ with $\phi_n(S)=S_n$. We say that $G$ is a \textit{mother group} of $(G_n)_n$ when we do not indicate $S$ and $(S_n)_n$.
\end{definition}

Fix $\mathcal{X}$. Let $(\Gamma_n)_{n\in\mathbb{N}}$ be a sequence of finite graphs with uniformly bounded degree and $\lim_{n\to \infty} |V(\Gamma_n) |\to \infty$. We call it (a sequence of) \textit{$\mathcal{X}$-panders} if for every $X\in \mathcal{X}$, $\inf_{n}\lambda(\Gamma_n,X)>0$, where  $\lambda (\Gamma,X)$ for finite $\Gamma=(V(\Gamma),E(\Gamma))$ is defined by
\[
\lambda (\Gamma,X)=\frac{1}{2}\inf_{f\colon V\to X} \frac{|V(\Gamma)|^2 \sum_{e\in E(\Gamma)}d(f(e^{+}),f(e^{-}))^2}{|E(\Gamma)|\sum_{(v,w)\in V(\Gamma)^2}d(f(v),f(w))^2}.
\]
Here $f$ runs over all non-constant maps, and $E(\Gamma)(\ni e=(e^{-},e^{+}))$ is the set of (oriented) edges.  This condition states that a \textit{Poincar\'{e}-type inequality holds for every $\Gamma_n$ with a common constant}. Ordinary expanders are $\mathbb{R}$-panders, which is equivalent to $\mathcal{B}_{L_q}$-panders for each $q\in [1,\infty)$ (see a book \cite{ostrovskii} for expanders). One of main importance of $\mathcal{X}$-panders is their \textit{lack} of ``\textit{$($uniform$)$ coarse embeddings} into $X$'' for every $X\in \mathcal{X}$ (see \cite[Definition~1.46]{ostrovskii}).

\begin{definition}[Super-expanders]\label{definition=superexpanders}
Define the class $\mathcal{B}_{\mathrm{uc}}$ be the class of all \textit{uniformly convex} Banach spcaces. Here, a Banach space $E$ is said to be uniformly convex if there exists $\delta=\delta_E \colon (0,2]\to \mathbb{R}_{>0}$ such that for all $\xi\neq \eta $ with $\|\xi\|=\|\eta\|=1$, $1 -(\|\xi+\eta\|/2)\geq \delta(\|\xi-\eta\|)$ holds. \textit{Super-expanders} are defined to be $\mathcal{B}_{\mathrm{uc}}$-panders.
\end{definition}

A direct generalization of Margulis's argument (\cite[Proposition~5.24]{ostrovskii}), \cite[Theorem~1.3.(1)]{BFGM} and \cite[Corollary 4.10]{cheng} show the following:

\begin{theorem}\label{theorem=superexpanders}
Let $\mathcal{E}$ be a class of Banach spaces. Assume there exists $q\in [1,\infty)$ such that for every $E\in \mathcal{E}$, $\ell_q(\mathbb{N},E)\in \mathcal{E}$. Assume $((G_n,S_n))_{n\in \mathbb{N}}$ has a mother group with property $(\mathrm{F}_{\mathcal{E}})$. Then, the sequence of Cayley graphs $(\mathrm{Cay}(G_n,S_n))_n$ forms $\mathcal{E}$-panders. Here the $\mathrm{Cayley\ graph}$ of $(H,T)$ is defined as $V=H$ and $E=\{(h,ht):h\in H,t\in T\}$. In particular, if $((G_n,S_n))_{n\in \mathbb{N}}$ has a mother group with property $(\mathrm{F}_{\mathcal{B}_{\mathrm{uc}}})$, then $(\mathrm{Cay}(G_n,S_n))_n$ forms super-expanders.
\end{theorem}
V. Lafforgue \cite{lafforgue} showed a lattice in $\mathrm{SL}(3,\mathbb{F})$, where $\mathbb{F}$ is a non-archimedean local field, has property $(\mathrm{F}_{\mathcal{B}_{\mathrm{uc}}})$. Hence, we may construct super-expanders, for instance, from (finite group quotients of) $\mathrm{SL}(3,\mathbb{F}_p[t])$ for $p$ prime. Another source of super-expanders was given by Mendel and Naor \cite{MN} by combinatorial construction. However, to the best knowledge of the author, it might be open whether we may construct super-expanders from finite simple groups of Lie type.  

A major motivation to study property $(\mathrm{F}_{\mathcal{B}_{\mathrm{uc}}})$ for $\mathrm{E}(n,R)$ for \textit{non-commutative} $R$ is the following (it was that of \cite{EJ} in relation to ordinary expanders).

\begin{lemma}[From elementary groups to super-expanders]\label{lemma=superexpanders}
\begin{enumerate}[$(1)$]
  \item If $\mathrm{E}(4,\mathbb{F}_p\langle s,t\rangle)$ has property $(\mathrm{F}_{\mathcal{B}_{\mathrm{uc}}})$ for each prime $p$, then $(1)$ of Conjecture~$\ref{conjecture=superexpanders}$ is true.
  \item If $\mathrm{E}(4,\mathbb{Z}\langle s,t\rangle)$ has property $(\mathrm{F}_{\mathcal{B}_{\mathrm{uc}}})$, then $(2)$ of Conjecture~$\ref{conjecture=superexpanders}$ is true.
\end{enumerate}
\end{lemma}

\begin{proof}
We only prove $(1)$. We concentrate on $\mathrm{SL}(4m,\mathbb{F}_{p})$ (if $n$ is not divisible by $4$, we need slight modification). Then, by considering $m\times m$ block matrices, we have a natural isomorphism $\mathrm{SL}(4m,\mathbb{F}_{p}) \simeq \mathrm{E}(4,\mathrm{Mat}_{m\times m}(\mathbb{F}_p))$. Observe that, for every $m$, the ring $\mathrm{Mat}_{m\times m}(\mathbb{F}_p)$ is generated by the ring unit $I_m$, $S_m:=e_{1,2}(1)$, and $T_m$, where $T_m$ is a permutation matrix of order $m$. Since $pI_m=pS_m=pT_m=O_m$, $\mathrm{E}(4,\mathbb{F}_p\langle s,t\rangle)$ \textit{is a mother group of}$(\mathrm{SL}(4m,\mathbb{F}_{p}))_m$. Theorem~\ref{theorem=superexpanders} ends our proof. 
\end{proof}

\noindent
\textbf{(3) Robust Banach property $(\mathrm{T})$.} Let $\mathcal{E}$ be a class of complex Banach spaces. Oppenheim \cite[Definition~1.2]{oppenheim}, inspired by \cite{lafforgue}, defined a notion of \textit{robust Banach property $(\mathrm{T})$ with respect to $\mathcal{E}$}, hereafter we write as \textit{robust property $(\mathrm{T}_{\mathcal{E}})$}. This is stated in terms of existence of \textit{Kazhdan-type projections} (in a weak sense) in certain Banach $\ast$-algebras. In work in progress, Gomez-Aparicio, Liao, and de la Salle \cite{GLS} defined a notion of \textit{geometric robust property $(\mathrm{T}_{\mathcal{E}})$} for a sequence of finite graphs. This is inspired by \textit{geometric property $(\mathrm{T})$} of Willett and Yu \cite{WY}. 
\begin{theorem}\label{theorem=robust}
  \begin{enumerate}[$(1)$]
   \item $($Oppenheim, \cite[Proposition~1.9]{oppenheim}$)$ Assume that there exists $q\in (1,\infty)$ such that for every $E\in \mathcal{E}$, $E\oplus_{\ell_q}\mathbb{C}\in \mathcal{E}$. Then, robust property $(\mathrm{T}_{\mathcal{E}})$ implies property $(\mathrm{F}_{\mathcal{E}})$.
   \item $($de la Salle, \cite[Corollary~5.11]{delasalle}$)$ Conversely, assume that $\mathcal{E}$ consists of \textit{superreflexive} Banach spaces, and that $\mathcal{E}$ satisfies $(\mathrm{U})$ in Subsection~$\ref{subsection=assumption}$. Then, $($for finitely generated groups$)$,  property $(\mathrm{F}_{\mathcal{E}})$ implies robust property $(\mathrm{T}_{\mathcal{E}})$. Here, a Banach space $E$ is $\mathrm{superreflexive}$ if and only if $E$ is isomorphic to some uniformly convex Banach space $($\cite[Theorem~A.6]{BL}$)$. 
\item $($Gomez-Aparicio--Liao--de la Salle, \cite{GLS}$)$ Let $\mathcal{E}$ as in $(2)$. Assume, besides, that there exists $q\in (1,\infty)$ such that for every $(E_n)_{n\in \mathbb{N}}\subseteq \mathcal{E}$, $(\sum_n E_n)_{\ell_q}\in \mathcal{E}$. Let $((G_n,S_n))_{n\in \mathbb{N}}$ have a mother group $(G,S)$. Then, $(\mathrm{Cay}(G_n,S_n))_n$ has geometric robust property $(\mathrm{T}_{\mathcal{E}})$ if $G$ has  robust property $(\mathrm{T}_{\mathcal{E}})$. In addition, the converse is  true if $\bigcap_{n}\mathrm{Ker}\{G\twoheadrightarrow G_n\}=\{1_G\}$.
\end{enumerate}
\end{theorem}
Examples of $\mathcal{E}$ with assumptions of $(1)$ and $(3)$ are $\mathcal{B}_{\mathrm{NC}L_q}$ (and $\mathcal{B}_{L_q}$) for every $q\in (1,\infty)$ ($(\mathrm{U})$ is due to Raynaud (and Heinrich)). 

\subsection{Main results and applications}\label{subsection=main}
As we mentioned in Subsection~\ref{subsection=heart},  we impose six assumptions on $\mathcal{X}$, which includes $(\mathrm{TP})$; and three hypotheses on group structures. We will see them in Section~\ref{section=hypotheses}.

\begin{mtheorem}[\textit{Superintrinsic synthesis}, two-subgroup case]\label{mtheorem=main1}
Assume that a class $\mathcal{X}$ of complete metric spaces satisfies   six assumptions $(\mathrm{S}_{+})$, $(\mathrm{U})$, $(\mathrm{NPC}1)$--$(\mathrm{NPC}3)$, and $(\mathrm{TP})$, which are defined in Subsection~$\ref{subsection=assumption}$. Let $M,L\leqslant G$, and $\Pi \leqslant \mathrm{Aut}(G)$. Assume that the following three hypotheses on $(G,M,L,\Pi)$ are fulfilled:

\noindent
\underline{$\mathrm{Hypotheses}$:}
\begin{enumerate}
 \item[$(i)$] the union $M\cup L$ generates $G$;
 \item[$(ii)$] the subgroup $\Pi$ is finite; and 
 \item[$(\mathrm{GAME})$] the player has a winning strategy for the $(\mathrm{Game})$ for $(M,L,\Pi)$, which is described in Subsection~$\ref{subsection=hypothesis}$.
\end{enumerate}

Then,  relative properties $(\mathrm{F}_{\mathcal{X}})$ for $G\geqslant M$ and $G\geqslant L$ imply property $(\mathrm{F}_{\mathcal{X}})$ for $G$.
\end{mtheorem}

\begin{mtheorem}[\textit{Superintrinsic synthesis}, two-subgroup case, special version]\label{mtheorem=main1b}
Let $G,M,L,\Pi$ and $\mathcal{X}$ be as in Theorem~\ref{mtheorem=main1}. Assume, besides, that $\mathcal{X}$ either consists of Banach spaces or of $\mathrm{CAT}(0)$ spaces. Assume that $(G,M,L,\Pi)$ fulfills $(i)$, $(ii)$ as in Theorem~\ref{mtheorem=main1}, and the following $\mathrm{hypothesis}$ $(\mathrm{GAME}^{+})$, which is a weaker hypothesis than $(\mathrm{GAME})$.
\begin{enumerate}
 \item[$(\mathrm{GAME}^{+})$] the player has a winning strategy for the $(\mathrm{Game}^{+})$ for $(M,L,\Pi)$, which is described in Subsection~$\ref{subsection=hypothesis}$.
\end{enumerate}
Then, relative properties $(\mathrm{F}_{\mathcal{X}})$ for $G\geqslant M$ and $G\geqslant L$ imply property $(\mathrm{F}_{\mathcal{X}})$ for $G$.
\end{mtheorem}
Recall from our notation that $G$ is assumed to be finitely generated. See Example~\ref{example=npc} for such $\mathcal{X}$, and Example~\ref{example=st} for such quadruples $(G,M,L,\Pi)$. 

\begin{remark}\label{remark=2sub}
In general, we may take $M_1,\ldots ,M_l$ for each finite $l$ instead of $M$ and $L$. See Theorem~\ref{mtheorem=main2} in Subsection~\ref{subsection=general} for the statement.
\end{remark}

\begin{remark}\label{remark=sr}
If $\mathcal{X}$ consists of Banach spaces, then by $(\mathrm{NPC}1)$ and $(\mathrm{U})$, automatically they are uniformly convex (see Remark~\ref{remark=bnpc}). In Theorem~\ref{mtheorem=sr}  in Subsection~\ref{subsection=sr}, we extend to the case where they are \textit{superreflexive} (recall from Theorem~\ref{theorem=robust}.$(2)$).
\end{remark}

On fixed point properties for elementary groups, these theorems complete the ``Synthesis Step'' for \textit{all} $\mathcal{X}$ with the six assumptions above. In fact, this synthesis works for more universal groups, which are called \textit{Steinberg groups}. They are written as $\mathrm{St}(n,R)$, and defined as  the group with generators $\{E_{i,j}(r):i\ne j \in [n],\ r\in R\}$ subject to the following three types of relations: for all $r_1,r_2\in R$,
\begin{itemize}
  \item for all $i\ne j$, $E_{i,j}(r_1)E_{i,j}(r_2)=E_{i,j}(r_1+r_2)$;
  \item for all $i\ne j$, $k\ne l$ with $i\ne l$ and $j\ne k$, $[E_{i,j}(r_1),E_{k,l}(r_2)]=1_{\mathrm{St}(n,R)}$; and 
  \item for all $i\ne j\ne k\ne i$, $[E_{i,j}(r_1),E_{j,k}(r_2)]=E_{i,k}(r_1r_2)$.
\end{itemize}
One importance of $\mathrm{St}(n,R)$ is that it is \textit{finitely presented},  provided that $R$ is finitely presented, and that $n\geq 4$ \cite{KM}. Note that the map $E_{i,j}(r)\mapsto e_{i,j}(r)$ gives rise to a natural homomorphism $\mathrm{St}(n,R)\twoheadrightarrow \mathrm{E}(n,R)$. Because fixed point properties pass to group quotients, statements on $\mathrm{St}(n,R)$ is stronger than those on $\mathrm{E}(n,R)$. Recall from Subsection~\ref{subsection=motivations}.$(1)$ the definition of $\mathcal{B}_{\mathrm{NC}L_q}$ and $\mathcal{B}_{\mathrm{Q}L_q}$ (and $\mathcal{B}_{L_q}$). Recall from our notation that $R$ is assumed to be unital and associative.

\begin{corollary}[Synthesis for Steinberg groups]\label{corollary=st}
Let $R$ be finitely generated and $n\geq 3$. Set $\tilde{G}=\mathrm{St}(n,R)$, $\tilde{M}=\langle E_{i,n}(r):i\in[n-1],r\in R\rangle$, and $\tilde{L}=\langle E_{n,j}(r):j\in[n-1],r\in R\rangle$.
\begin{enumerate}[$(a)$]
  \item $($Conditional results$)$ Let $\mathcal{X}$ satisfy the six assumptions as in Theorem~\ref{mtheorem=main1}. Then, for every $n\geq 4$, relative properties $(\mathrm{F}_{\mathcal{X}})$ for $\tilde{G}\geqslant \tilde{M}$ and for $\tilde{G}\geqslant \tilde{L}$ imply property $(\mathrm{F}_{\mathcal{X}})$ for $\tilde{G}$. If, besides, $\mathcal{X}$ consists of $\mathrm{CAT}(0)$ spaces, then the same holds true for $n=3$.

Let $\mathcal{E}$ be a class of superreflexive Banach spaces that satisfies  $(\mathrm{U})$. Then, for every $n\geq 3$, relative properties $(\mathrm{F}_{\mathcal{E}})$ for $\tilde{G}\geqslant \tilde{M}$ and for $\tilde{G}\geqslant \tilde{L}$ imply property $(\mathrm{F}_{\mathcal{E}})$ for $\tilde{G}$.
  \item $($Unconditional results$)$ 
   \begin{enumerate}[$(1)$]
        \item $($Ershov--Jaikin-Zapirain \cite[Theorem 6.2]{EJ}$)$ For every $n\geq 3$ and for every $R$, $\mathrm{St}(n,R)$ has property $(\mathrm{F}_{\mathcal{H}\mathrm{ilbert}})$.
        \item For every $n\geq 4$ and for every $R$, $\mathrm{St}(n,R)$ has property $(\mathrm{F}_{\mathcal{B}_{\mathrm{NC}L_q}})$ for $\mathrm{all}$ $q\in (1,\infty)$. It has property $(\mathrm{F}_{\mathcal{B}_{\mathrm{Q}L_q}})$ for $\mathrm{all}$ $q\in (1,\infty) \setminus \{\frac{2j}{2j-1}: j\in \mathbb{Z}_{\geq 2}\}$.
        \item For every $n\geq 4$ and for every $R$, $\mathrm{St}(n,R)$ has property $(\mathrm{F}_{[\mathcal{H}\mathrm{ilbert}]})$. Here $[\mathcal{H}\mathrm{ilbert}]$ is the class of Banach spaces  isomorphic to Hilbert spaces.
\end{enumerate}
\end{enumerate}
\end{corollary}
Recall  that isomorphisms between Banach spaces  mean linear ones. Property $(\mathrm{F}_{[\mathcal{H}\mathrm{ilbert}]})$ is equivalent to the fixed point property with respect to affine, possibly non-isometric, actions on Hilbert spaces, whose linear part $\rho$ are \textit{uniformly bounded} ($\sup_{\gamma \in G}\|\rho(\gamma)\|_{\mathrm{op}}<\infty$). See \cite[Proposition 2.3]{BFGM} for this equivalence. The \textit{Banach--Mazur distance} between two isomorphic Banach spaces $E$ and $F$ is defined by the infimum of $\|T\|_{\mathrm{op}}\|T^{-1}\|_{\mathrm{op}}$, where $T$ runs over all isomorphisms $E\stackrel{\simeq}{\to} F$. For every $C\geq 1$, define the class $[\mathcal{H}\mathrm{ilbert}]_C(\subseteq [\mathcal{H}\mathrm{ilbert}])$ of all Banach spaces with Banach-Mazur distance from a Hilbert space at most $C$. Note that it satisfies $(\mathrm{U})$.

The emphasis on  Corollary~\ref{corollary=st} is, as mentioned before, \textit{extrinsic} synthesis may have some difficulty in establishing results with respect to a class of ``\textit{unbounded} wildness'', such as $(b)$.$(2)$. In addition, if $R$ is not an algebra over a finite field (for instance, if $R=\mathbb{Z}\langle s,t\rangle$), then known results by extrinsic syntheses may be weaker than $(b).$$(2)$--$(3)$ above. We mention, on the other hand, if $R$ is an algebra over $\mathbb{F}_p$ for sufficiently large prime $p$, then Oppenheim \cite[Theorem~1.11 and Remark~5.7]{oppenheim} provided a considerably strong result (but in ``bounded wildness'').

\begin{remark}\label{remark=l1}
Combination of \cite[Theorem~1.3]{BFGM} and \cite[Theorem~A]{BGM} shows that for every $1\leq q\leq2$, property $(\mathrm{F}_{\mathcal{B}_{L_q}})$ is, in fact, equivalent to property $(\mathrm{F}_{\mathcal{H}\mathrm{ilbert}})$. However, the key to this proof (\cite[Subsection~3.b]{BFGM}) does not hold for non-commutative $L_q$-spaces. Therefore, the assertion of Corollary~\ref{corollary=st}.$(b)$.$(2)$ is \textit{new} even for $1<q<2$. For the case $q=1$, we discuss in Remark~\ref{remark=noncommutativel1}. For property $(\mathrm{F}_{\mathcal{B}_{\mathrm{Q}L_1}})$, $G$ has it if and only if $G$ is finite because $(\sum_{n\geq 1} \ell_{2n})_{\ell_2} \in \mathcal{B}_{\mathrm{Q}L_1}$.
\end{remark}
Corollary~\ref{corollary=st} provides the following byproducts. 

\begin{corollary}\label{corollary=expanders}
\begin{enumerate}[$(1)$]
  \item Let $\tilde{G}=\mathrm{St}(n,R)$, where $n\geq 4$ and $R$ finitely generated. Let $G$ be a group quotient of a finite index subgroup of $\tilde{G}$, and $\Lambda$ be a group that is $\ell_q$-measure equivalent to $G$ for all $q\in (1,\infty)$ $($see a survey \cite[Definition~2.1 and Subsection~2.3.2]{furman} for the definitions$)$. Then, $\Lambda$ is $\mathrm{never}$ relatively hyperbolic.
  \item For every sequence of primes $(p_n)_{n\geq 4}$, the sequence $(\mathrm{SL}(n,\mathbb{F}_{p_n}))_{n\geq 4}$ can form expanders with geometric robust property $(\mathrm{T}_{\mathcal{E}})$. Here, $\mathcal{E}=\mathcal{B}_{\mathrm{NC}L_q}$ for each $q\in (1,\infty)$; $\mathcal{E}=\mathcal{B}_{\mathrm{Q}L_q}$ for each $q\in (1,\infty)\setminus \{\frac{2j}{2j-1}:j\in \mathbb{Z}_{\geq 2}\}$; and $\mathcal{E}=[\mathcal{H}\mathrm{ilbert}]_C$ for every $C\geq 1$.
\end{enumerate}
\end{corollary}

\begin{corollary}[\textit{Reduction of $($part of$)$ Unbounded Rank Super-expander Conjecture to relative  property $(\mathrm{T}_{\mathcal{B}_{\mathrm{uc}}})$}]\label{corollary=superexpanders}
Set $A_p=\mathbb{F}_p\langle s,t\rangle$ for each prime $p$ and $A=\mathbb{Z}\langle s,t\rangle$.
\begin{enumerate}[$(1)$]
  \item If $\mathrm{E}(2,A_p)\ltimes A_p^2 \trianglerighteq A_p^2$ has relative  property $(\mathrm{T}_{\mathcal{B}_{\mathrm{uc}}})$ with respect to the standard finite set $($see Definition~$\ref{definition=relbanach}$ and Definition~$\ref{definition=standard}$, respectively, for definitions of relative property $(\mathrm{T}_{\mathcal{E}})$ and standard finite subsets$)$, then $\mathrm{E}(n,A_p)$ has property $(\mathrm{F}_{\mathcal{B}_{\mathrm{uc}}})$ for all $n\geq 4$. In particular, then, $(1)$ of Conjecture~$\ref{conjecture=superexpanders}$ will be resolved in the affirmative for such $p$.
  \item If $\mathrm{E}(2,A)\ltimes A^2 \trianglerighteq A^2$ has relative property $(\mathrm{T}_{\mathcal{B}_{\mathrm{uc}}})$ with respect to the standard finite set, then $\mathrm{E}(n,A)$ has property $(\mathrm{F}_{\mathcal{B}_{\mathrm{uc}}})$ for all $n\geq 4$. In particular, then, $(2)$ of Conjecture~$\ref{conjecture=superexpanders}$ will be resolved in the affirmative.
\end{enumerate}
\end{corollary}

Steinberg groups $\mathrm{St}(n,R)$ are ones associated with root system $A_{n-1}$. For other root systems case, see Remark~\ref{remark=root}.

\section{Six assumptions and $(\mathrm{GAME})$ hypothesis}\label{section=hypotheses}
\subsection{Six assumptions on $\mathcal{X}$}\label{subsection=assumption}
The following are the \textit{six assumptions} in Theorem~\ref{mtheorem=main1}. Recall from Definition~\ref{definition=parallelism} the definition of parallelism ($||$).
\begin{itemize}
  \item $(\mathrm{S}_{+})$ [stability under \textit{scaling up}]: for $(X,d)\in \mathcal{X}$ and for  $r\geq 1$, $(X,rd)\in \mathcal{X}$.
  \item $(\mathrm{U})$ [stability under metric \textit{ultraproducts}]: there exists a non-principal ultrafilter $\mathcal{U}$ such that for all $X_n \in \mathcal{X}$ and for all $w_n\in X_n$, the pointed metric ultraproduct $\lim_{\mathcal{U}}(X_n,w_n)$ belongs to $\mathcal{X}$ (after forgetting the base point $[(w_n)_n]$). (See Subsection~\ref{subsection=up} for basics on metric ultraproducts.)
  \item $(\mathrm{NPC})$ [\textit{non-positively curved}]: 
\begin{itemize}
   \item $(\mathrm{NPC}1)$: the space $X$ is a uniquely geodesic space (this means that, every pair of points is connected by a unique geodesic segment), and for all two geodesic segments $[x,y]$ and $[x',y']$ such that $d(x,x')=d(y,y')$,  and for all $0\leq t\leq 1$, $d(m_{t}[x,y],m_{t}[x',y'])\leq d(x,x')$.
    \item $(\mathrm{NPC}2)$: if $[x,y]\, ||\,[x',y']$, then $[x,x']\, ||\,[y,y']$.
   \item $(\mathrm{NPC}3)$: if $[x,y]\, ||\,[y,z]$, then $[x,y]\cup [y,z]$ (the concatenation of $[x,y]$ and $[y,z]$) is a geodesic.
\end{itemize}
  \item $(\mathrm{TP})$ [\textit{transitive law of the parallelism}]: $(2)$ of Definition~\ref{definition=parallelism}.
\end{itemize}

Sometimes, we replace $(\mathrm{NPC}1)$ with stronger assumption $(\Theta$-$\mathrm{NPC}1)$: for a (fixed)  strictly increasing continuous function $\Theta\colon \mathbb{R}_{\geq 0}\to \mathbb{R}_{\geq 0}$, we say that $X$ satisfies $(\Theta$-$\mathrm{NPC}1)$ if for all pairs $[x,y]$ and $[x',y']$ of geodesic segments in $X$ and for all $0\leq t\leq 1$, 
\[
\Theta(d(m_t[x,y], m_t[x',y']))\leq (1-t)\Theta(d(x,x'))+t\Theta(d(y,y')).
\]
By abuse of notation, we say that $\mathcal{X}$ satisfies, respectively, $(\mathrm{NPC}i)$ ($i=1,2,3$) and $(\Theta$-$\mathrm{NPC}1)$ if every $X\in \mathcal{X}$ satisfies the corresponding property. 

\begin{remark}\label{remark=bnpc}
All \textit{BNPC}  spaces satisfy $(\mathrm{NPC}1)$--$(\mathrm{NPC}3)$. See \cite[Theorem~3.14]{busemann}. Here, the Busemann NPC condition coincides with $(\Theta$-$\mathrm{ NPC}1)$ assumption for $\Theta(r):=r$. For a Banach space, the BNPC condition exactly corresponds to the strict convexity. 
Due to $(\mathrm{U})$, we automatically impose certain ``uniformity'' on NPC-conditions on $X\in \mathcal{X}$. For instance, for a class of Banach spaces $\mathcal{E}$ with $(\mathrm{U})$ and $(\mathrm{NPC}1)$, every $E\in \mathcal{E}$ is  uniformly convex. This is because a metric ultraproduct of a (fixed) Banach space $E$ is strictly convex if and only if $E$ is uniformly convex  (recall from Definition~\ref{definition=superexpanders} the definition of uniform convexity). 
\end{remark}

\begin{remark}\label{remark=tp}
We warn that the assumption $(\mathrm{TP})$ is strong. For instance, for two unit squares $\mathrm{ABCD}$ and $\mathrm{EFGH}$, glue the triangles $\mathrm{BCD}$ and $\mathrm{FGH}$. In this way, we may obtain a $\mathrm{CAT}(0)$ space $\mathrm{ABCD}$-$\mathrm{E}$. Then, there, $[\mathrm{A},\mathrm{B}]\,||\,[\mathrm{D},\mathrm{C}]$ and $[\mathrm{D},\mathrm{C}]\,||\,[\mathrm{E},\mathrm{B}]$. Nevertheless, $[\mathrm{A},\mathrm{B}]$ and $[\mathrm{E},\mathrm{B}]$ are \textit{not} parallel. 

However, a BNPC space $X$ under the following two conditions satisfies $(\mathrm{TP})$:
\begin{itemize}
  \item the space $X$ is \textit{uniquely geodesically $($bi-$)$complete}. That means, every geodesic segment is uniquely extendable to a bi-infinite geodesic line; and 
  \item if $[x,y]\,||\,[x',y']$, then the two (uniquely) extended geodesic lines are \textit{asymptotic} (this means, the distance function for these two geodesics is bounded).
\end{itemize}

For instance, a real-analytic Hadamard manifold satisfies these two conditions because, then, for two geodesic segments, the square of distance function for them is real-analytic. We refer to \cite{GKM} for further discussions.
\end{remark}

\begin{example}[Examples of $\mathcal{X}$ that satisfies these six assumptions]\label{example=npc}
\begin{enumerate}[$(1)$]
  \item The class $\mathcal{B}_{L_q}$ and $\mathcal{B}_{\mathrm{NC}L_q}$ for all $1<q<\infty$ (see below Theorem~\ref{theorem=robust}). Not true for $q=1$. Also, $\mathcal{B}_{\mathrm{Q}L_q}$ for $q\in (1,\infty)$. Indeed,  for Banach spaces $(E_n)_n$ and subspaces $(F_n)_n$, the ultraproduct $\lim_{\mathcal{U}}(E_n/F_n)$ is naturally isometric to $(\lim_{\mathcal{U}}E_n)/(\lim_{\mathcal{U}}F_n)$. 
  \item The class $\mathcal{A}\mathbb{R}$ of all $\mathbb{R}$-trees.
  \item In general, fix a metric space $(X,d)$. Consider the class $\mathcal{T}_X$ of all metric ultraproducts of $((X,r_nd,w_n))_{n\in \mathbb{N}}$ for all $r_n\geq 1$ and $w_n\in X$ with respect to \textit{all} ultrafilters  (after forgetting the basepoint)  satisfy $(\mathrm{S}_{+})$ and $(\mathrm{U})$. For $(\mathrm{U})$, see \cite[Proposition~3.23 and Corollary~3.24]{DS}. Note that we here allow also principal ultrafilter, so that $(X,rd) \in \mathcal{X}$ for all $r\geq 1$. If all of the elements in $\mathcal{T}_{X}$ satisfies $(\mathrm{NPC}1)$--$(\mathrm{NPC}3)$ and $(\mathrm{TP})$, then $\mathcal{T}_{X}$ satisfies the six assumptions.
  \item An example in direction to $(3)$: The class $\mathcal{T}_{E}$ for a uniformly convex Banach space $E$. 
  \item Another example in that direction is $\mathcal{T}_{X}$ for $X$ a symmetric  space of non-compact type. It follows from Remark~\ref{remark=tp}, because every element in $\mathcal{T}_{X}$ is either an upscaling of $X$, or a Euclidean space. Here it is crucial that in $(\mathrm{S}_{+})$, we only require stability of scaling \textit{up} to avoid having asymptotic cones, which may admit branching of geodesics. More generally, $\mathcal{T}_{X}$ for $X$ a real-analytic homogeneous Hadamard manifold works.
\end{enumerate}
\end{example}

We roughly describe roles of these assumptions (compare with Subsection~\ref{subsection=heart}).
\begin{itemize}
  \item $(\mathrm{S}_+)$ and $(\mathrm{U})$: to ensure the ``existence of realizers'' in a certain situation. See Section~\ref{section=up}.
  \item $(\mathrm{NPC}1)$--$(\mathrm{NPC}3)$ and $(\mathrm{TP})$: to prove ``Pseudo-Uniqueness'' of such realizers. See Section~\ref{section=parallel}.
  \item $(\Theta$-$\mathrm{NPC}1)$, $(\mathrm{NPC}2)$, $(\mathrm{NPC}3)$ and $(\mathrm{TP})$: to prove ``Pseudo-Uniqueness of realizers''  for general cases (corresponding to $M_1,\ldots ,M_l$). See Subsection~\ref{subsection=generalproof}.
\end{itemize}

\subsection{The $(\mathrm{Game})$ and $(\mathrm{Game}^{+})$}\label{subsection=hypothesis}
First, we present the definition of $(\mathrm{Game}^{+})$. This is a formalization of the rough \textit{self-improvement argument} appearing in Proposition~\ref{proposition=heart} in an extended way.

The ``\textit{$(\mathrm{Game}^{+})$ for $(M,L,\Pi)$}'' is a one-player game. We fix $M,L\leqslant G$ and $\Pi\leqslant \mathrm{Aut}(G)$, and keep them unchanged. We take $H_1,H_2\leqslant G$. These two subgroups vary stage by stage in the game. 

\noindent
(Rules of $(\mathrm{Game}^{+})$:)
\begin{itemize}
  \item in the initial stage, $H_1=M$ and $H_2=L$;
  \item in each stage, the player is allowed to enlarge both of $H_1$ and $H_2$ by admissible moves. The admissible moves consist of type $(\mathrm{I}^{+})$ moves and type $(\mathrm{I\hspace{-.1em}I})$ moves. Type $(\mathrm{I\hspace{-.1em}I})$ moves are indexed by elements in $\mathrm{Sym}(2)=\{\mathrm{id},(12)\}$. Hence, they are of type $(\mathrm{I\hspace{-.1em}I}_{\mathrm{id}})$, and of type $(\mathrm{I\hspace{-.1em}I}_{(12)})$. We define each move below; and
  \item the winning condition is that within finite steps of moves, the player sets either $H_1=G$ or $H_2=G$.
\end{itemize}
Rough meaning of these rules is the following. Assume $X\in \mathcal{X}$ and $\alpha \colon G\curvearrowright X$. Take a pair $(x,y)$, where $x\in X^{\alpha(M)}$ and $y\in X^{\alpha(L)}$ in a certain manner similar to one in Proposition~\ref{proposition=heart}. Then, (under $(i)$ and $(ii)$ as in Theorem~\ref{mtheorem=main1}), \textit{$H_1$ and $H_2$ in each stage, respectively, represent the subgroups such that we know that $x\in X^{\alpha(H_1)}$ and $y\in X^{\alpha(H_2)}$ at that stage.} Therefore, the winning condition above exactly corresponds to that either of $x$ or $y$ lies in $X^{\alpha(G)}$. From these backgrounds, the following formulations of moves may not be too intricate.

\noindent
(Rules of the moves in $(\mathrm{Game}^{+})$ for $(M,L,\Pi)$.)
\begin{itemize}
  \item \textit{Type $(\mathrm{I}^{+})$ move}: pick a subset $P\subseteq G$ such that for all $h\in P$,
\[
hH_1h^{-1}\geqslant M \quad \mathrm{and}\quad hH_2h^{-1}\geqslant L.
\]
Then, enlarge $H_1$ and $H_2$ as:
\begin{table}[h]
 \begin{tabular}{c|c}
   $H_1$  & $H_2$ \\ \hline 
$\langle H_1,P\rangle$ & $\langle H_2,P\rangle$.
  \end{tabular}
\end{table}

Here, by the table above, we denote that we enlarge $H_1$ to $\langle H_1,P\rangle$, and that set it as the new $H_1$ at the stage after this move; and similarly on $H_2$. In the other moves below, we use similar tables to indicate such enlargements.
  \item \textit{Type $(\mathrm{I\hspace{-.1em}I}_{\sigma})$ move} for $\sigma \in \mathrm{Sym}(2)$: pick a subset $\Lambda\subseteq \mathrm{Inn}(G)\cdot \Pi(\leqslant \mathrm{Aut(G)})$ such that for all $\phi\in \Lambda$,
\[
H_{\sigma^{-1}(1)}\geqslant \phi(M)\quad \mathrm{and}\quad H_{\sigma^{-1}(2)} \geqslant \phi(L).
\]
Then, enlarge $H_1$ and $H_2$  as:
\begin{table}[h]
 \begin{tabular}{c|c}
   $H_1$  & $H_{\sigma(1)} $\\ \hline 
$\langle H_1,\bigcup_{\phi \in \Lambda}\phi(H_{\sigma(1)})\rangle$& $\langle H_2,\bigcup_{\phi \in \Lambda}\phi(H_{\sigma(2)})\rangle$.
  \end{tabular}
\end{table}
\end{itemize}
Note that $\sigma=\sigma^{-1}$ for $\sigma\in\mathrm{Sym}(2)$. We, nevertheless, use $\sigma^{-1}$ above because this is the right formulation in general case of $(M_1,\ldots,M_l,\Pi)$ (in ``$(\mathrm{Game}_l^{+})$'').

The first move in the proof of Proposition~\ref{proposition=heart} is of type $(\mathrm{I})$. The second move there is of type $(\mathrm{I\hspace{-.1em}I}_{(12)})$, where $\Lambda$ is a singleton of the inner conjugation of $w^{-1}$.

This $(\mathrm{Game}^+)$ is a relaxed version of $(\mathrm{Game})$ (``${}^+$'' indicates the relaxation). In general case, we need to impose one more condition on type $(\mathrm{I})$ moves.

\begin{definition}[$(\mathrm{Game})$ for $(M,L,\Pi)$]\label{definition=game}
We define \textit{type $(\mathrm{I})$ move} by imposing on $P$ in type $(\mathrm{I}^{+})$ move the additional condition that $Q^{\mathrm{abel}}$ is a torsion group. Here $Q:=\langle P\rangle$, and ${}^{\mathrm{abel}}$ means the abelianization. The \textit{$(\mathrm{Game})$} is defined by replacing type $(\mathrm{I}^+)$ moves with type $(\mathrm{I})$ moves in the definition of $(\mathrm{Game}^+)$.
\end{definition}

\begin{example}[Examples of quadruples $(G,M,L,\Pi)$]\label{example=st}
As we saw in Proposition~\ref{proposition=heart}, for $G=\mathrm{E}(n,R)$, $M=\langle e_{i,n}(r):i\in [n-1],r\in R\rangle$ and $L=\langle e_{n,j}(r):j\in [n-1],r\in R\rangle$, the quadruple $(G,M,L,\{1_{\mathrm{Aut}(G)}\})$ satisfies the hypothesis $(\mathrm{GAME}^{+})$ as in Theorem~\ref{mtheorem=main1b} for every $n\geq 3$. For $P=\langle e_{i,j}(r):i\ne j \in [n-1],r\in R\rangle(\simeq \mathrm{E}(n-1,R))$ in this example, $Q^{\mathrm{abel}}$ is trivial as long as $n\geq 4$ (by $(\ast)$). Hence, the quadruple $(G,M,L,\{1_{\mathrm{Aut}(G)}\})$ satisfies the hypothesis $(\mathrm{GAME})$ as in Theorem~\ref{mtheorem=main1} if $n\geq 4$.

If we switch to the case of Steinberg groups, then we employ $\Pi:=\langle \phi_{\tau}\rangle(=\{1_{\mathrm{Aut}(G)}, \phi_{\tau}\})$, where $\tau$ is the transposition of $\{1,n\}$ (inside $[n]$) and $\phi_{\tau}(E_{i,j}(r)):=E_{\tau(i),\tau(j)}(r)$ (this gives rise to an automorphism). Then, we may extend the argument in Proposition~\ref{proposition=heart} to this case, by replacing the inner conjugation of $w^{-1}$ with $\phi_{\tau}$. Thus, for $\tilde{G},\tilde{M}$, and $\tilde{L}$ as in Corollary~\ref{corollary=st}, \textit{the quadruple $(\tilde{G},\tilde{M},\tilde{L},\Pi)$ satisfies $(\mathrm{GAME}^+)$ for every $n\geq 3$; and it does $(\mathrm{GAME})$, provided that $n\geq 4$}.
\end{example}

It was  pointed out by de la Harpe that,  on self-improvement process, there might be some formal similarity to the \textit{Mautner phenomenon} on unitary representations of continuous (Lie) groups. See \cite[Lemma 1.4.8]{BHV}.

\subsection{Main theorem in full generality}\label{subsection=general}
\begin{mtheorem}[\textit{Superintrinsic synthesis}, general case]\label{mtheorem=main2}
Let $\mathcal{X}$ be a class that satisfies  assumptions $(\mathrm{S}_{+})$, $(\mathrm{U})$, $(\Theta$-$\mathrm{NPC}1)$, $(\mathrm{NPC}2)$, $(\mathrm{NPC}3)$, and $(\mathrm{TP})$ for some $\Theta$. Let $l\in \mathbb{N}_{\geq 3}$. Let $M_1,\ldots ,M_l\leqslant G$, and $\Pi \leqslant \mathrm{Aut}(G)$. Assume that the following three hypotheses on $(G,M_1,\ldots ,M_l,\Pi)$ are fulfilled:

\noindent
\underline{$\mathrm{Hypotheses}$:}
\begin{enumerate}
 \item[$(i_{l})$] the union $\bigcup_{i=1}^l M_i$ generates $G$;
 \item[$(ii)$] the subgroup $\Pi$ is finite; and
 \item[$(\mathrm{GAME}_l)$] the player has a winning strategy for the $(\mathrm{Game}_l)$ for $(M_1,\ldots ,M_l,\Pi)$.
\end{enumerate}

Then, relative properties $(\mathrm{F}_{\mathcal{X}})$ for $G\geqslant M_i$ 
for all $i$ imply property $(\mathrm{F}_{\mathcal{X}})$ for $G$.

In addition, if $\mathcal{X}$, besides, consists either of Banach spaces or of $\mathrm{CAT}(0)$ spaces, then hypothesis $(\mathrm{GAME}_l)$ may be relaxed to the following hypothesis $(\mathrm{GAME}_l^{+})$: ``the existence of winning strategies for $(\mathrm{Game}^+_l)$ for $(M_1,\ldots ,M_l,\Pi)$''.
\end{mtheorem}

\begin{definition}[$(\mathrm{Game}_l^+)$ and $(\mathrm{Game}_l)$ for $(M_1,\ldots ,M_l,\Pi)$]\label{definition=gamel}
The $(\mathrm{Game}_l^+)$ is again a one-player game. We fix $(M_1,\ldots, M_l,\Pi)$ and keep it unchanged. We set $H_1,\ldots ,H_l\leq G$, and they vary stage by stage. In the initial stage,  $H_i=M_i$ for $1\leq i\leq l$. In each stage, the player enlarges all of $H_i$'s by the moves of type $(\mathrm{I}_l^+)$ and $(\mathrm{I\hspace{-.1em}I}_l)$, which we define below. The winning condition is to set $H_i=G$ for some $i$.

Here type $(\mathrm{I}_l^+)$ moves are indexed by $\tau \in \mathrm{Sym}(l)$ that are \textit{not} derangement, and type $(\mathrm{I\hspace{-.1em}I}_l)$ moves are indexed by $\sigma \in \mathrm{Sym}(l)$. Recall that, a \textit{derangement} is a fixed-point-free permutation, and we denote $\mathrm{Der}(l)(\subseteq \mathrm{Sym}(l))$ by the set of them.

\noindent
(Rules of the moves in $(\mathrm{Game}_l^+)$ for $(M_1,\ldots,M_l,\Pi)$.)
\begin{itemize}
  \item \textit{Type $(\mathrm{I}_{\tau,l}^+)$ move} for $\tau \in \mathrm{Sym}(l)\setminus \mathrm{Der}(l)$: pick $P\subseteq G$ such that for all $h\in P$ and for all $1\leq i\leq l$, $hH_ih^{-1}\geqslant  M_{\tau(i)}$. Then, enlarge each $H_i$ as:
\begin{table}[h]
 \begin{tabular}{c|c}
   $H_i$  & $H_j$ \\ \hline 
$\langle H_i,P\rangle$ & $H_j$.
  \end{tabular}
\end{table}

Here, $i$ runs over all $i$ with $\tau(i)=i$, and $j$ runs over all other indices.
  \item \textit{Type $(\mathrm{I\hspace{-.1em}I}_{\sigma,l})$ move} for $\sigma \in \mathrm{Sym}(l)$: pick $\Lambda\subseteq \mathrm{Inn}(G)\cdot \Pi$ with $H_{i}\geqslant \phi(M_{\sigma(i)})$ for all $\phi\in \Lambda$ and for all $1\leq i \leq l$. Then, enlarge each $H_i$ as: for all $1\leq i\leq l$,
\begin{table}[h]
 \begin{tabular}{c}
   $H_i$   \\ \hline 
$\langle H_i,\bigcup_{\phi \in \Lambda}\phi(H_{\sigma(i)})\rangle$.
  \end{tabular}
\end{table}
\end{itemize}

For $(\mathrm{Game}_l)$, \textit{type $(\mathrm{I}_{l})$ moves} are defined by imposing on type $(\mathrm{I}_{l}^+)$ moves the condition that $Q^{\mathrm{abel}}$ is torsion for $Q:=\langle P\rangle$. The \textit{$(\mathrm{Game}_l)$} is defined by replacing all type $(\mathrm{I}_{l}^+)$ moves with type $(\mathrm{I}_{l})$ moves in $(\mathrm{Game}^+_l)$.
\end{definition}

\section{Self-improve argument for type $(\mathrm{I})$ moves, and Pseudo-Uniqueness}\label{section=parallel}
In this section, we present a rigorous statement, in the spirit of Proposition~\ref{proposition=heart}, for \textit{our self-improvement argument for type $(\mathrm{I})$ $($and $(\mathrm{I}^+))$ moves} for realizers $(x,y)$ of the distance. To obtain such self-improvement, we study \textit{Pseudo-Uniqueness} (recall Lemma~\ref{lemma=pseudounique}). The goal in this section is the following proposition.

\begin{proposition}[Key proposition: \textit{Self-improvement argument for type $(\mathrm{I})$ moves}]\label{proposition=key}

Assume that $X$ satisfies $(\mathrm{NPC}1)$, $(\mathrm{NPC}2)$, and $(\mathrm{TP})$. Let $G\geqslant M,L$ with $\langle M,L\rangle =G$. Let $\alpha\colon G\curvearrowright X$ with $X^{\alpha(M)}\ne \emptyset$ and $X^{\alpha(L)}\ne \emptyset$. Assume that $(x,y)$, where $x\in X^{\alpha(M)}$ and $y\in X^{\alpha(L)}$, realizes $D:=\mathrm{dist}(X^{\alpha(M)},X^{\alpha(L)})$. Let $H_1,H_2\leqslant G$ and $P\subseteq G$. Assume that $x\in X^{\alpha(H_1)}$ and $y\in X^{\alpha(H_2)}$; and that $hH_1h^{-1}\geqslant M$ and $hH_2h^{-1}\geqslant L$ for all $h\in P$. Set $Q:=\langle P\rangle$.

\begin{enumerate}[$(1)$]
  \item Then, the orbit map of $x$ by $Q$, $\Psi_x\colon Q\to X\colon g\mapsto g\cdot x$, factors through the abelianization map $Q\twoheadrightarrow Q^{\mathrm{abel}}$. The same holds for the orbit map of $y$ by $Q$.
  \item If, besides, $X$ is $\mathrm{CAT}(0)$, then $x\in X^{\alpha(\langle H_1,P\rangle)}$ and $y\in X^{\alpha(\langle H_2,P\rangle)}$.
  \item If, besides, $X$ is a uniformly convex Banach space and if $G^{\mathrm{abel}}$ is finite, then $x\in X^{\alpha(\langle H_1,P\rangle)}$ and $y\in X^{\alpha(\langle H_2,P\rangle)}$.
\end{enumerate}
In particular, if, besides, $Q^{\mathrm{abel}}$ is torsion and if $X$ satisfies $(\mathrm{NPC}3)$, then $x\in X^{\alpha(\langle H_1,P\rangle)}$ and $y\in X^{\alpha(\langle H_2,P\rangle)}$.
\end{proposition}

Recall from our notation that $\alpha$ is assumed to be isometric. We will prove $(3)$, $(1)$, $(2)$, respectively, in Subsections~\ref{subsection=banach}, \ref{subsection=bnpc}, and \ref{subsection=cat0}.

\subsection{The case of Banach spaces}\label{subsection=banach}
In the setting of Proposition~\ref{proposition=key}.$(3)$, we change the symbols $X,x,y$, respectively, to $E,\xi,\eta$. 

We may regard $E$ as a real Banach space. Then, $\alpha$ is an affine isometric action, as seen in Subsection~\ref{subsection=heart}, and $\alpha$ decomposes as $\alpha=\rho +b$, where $\rho$ is the linear part and $b$ is the cocycle part (recall from the paragraph above Lemma~\ref{lemma=pseudounique}). Then, by strict convexity,  the same argument as one in Lemma~\ref{lemma=pseudounique} applies, and  provides Pseudo-Uniqueness of realizers: for two realizers $(\xi,\eta)$ and $(\xi',\eta')$ of $D$, 
\[
\xi-\xi'=\eta-\eta'\in E^{\rho(G)}.
\]
\begin{proof}[Proof of $(3)$ of Proposition~$\ref{proposition=key}$]
This is a straightforward generalization of Shalom's argument \cite[4.III]{shalom2006}. Let $h\in P\cup P^{-1}$. By assumptions on $h$, $\alpha(h)\cdot \xi\in E^{\alpha(M)}$ and $\alpha(h)\cdot\eta\in E^{\alpha(L)}$. By isometry of $\alpha$, $(\alpha(h)\cdot \xi, \alpha(h)\cdot \eta)$ is another realizer of $D$. Therefore, our Pseudo-Uniqueness above implies that $\alpha(h)\cdot \xi -\xi=\alpha(h)\cdot \eta -\eta \in E^{\rho(G)}$. Then, we may replace $h\in P\cup P^{-1}$ with every $g\in Q$ in the formula above. 

By uniform convexity of $E$, by \cite[Proposition~2.6]{BFGM}, $E$ is decomposed into  $E=E^{\rho(G)}\oplus E_{\rho(G)}'$ as $G$-representations. We decompose as $b=b_1+b_2$, where $b_1\in E^{\rho(G)}$ and $b_2\in E_{\rho(G)}'$. Then $b_1$ and $b_2$ are also cocycles. Set $\alpha_1=\rho+b_1$ and $\alpha_2=\rho+b_2$. They are affine isometric $G$-actions, respectively, on $E^{\rho(G)}$ and on $E_{\rho(G)}'$. Observe that $b_1\equiv 0$ because $G^{\mathrm{abel}}$ is finite. This implies that $\alpha_1$ is a trivial action. We claim that $\alpha_2(g)\cdot \xi_2 =\xi_2$, where $\xi=\xi_1+\xi_2$ is a decomposition according to the decomposition of $E$ above. Indeed, this follows from that $\alpha_2(g)\cdot \xi_2 -\xi_2 \in E^{\rho(G)} \cap E_{\rho(G)}'=\{0\}$. Therefore,  $\alpha(g)\cdot \xi=\xi$ for all $g\in Q$. Similarly for $\eta$.
\end{proof}

\subsection{Pseudo-Uniqueness in general case, and parallelism}\label{subsection=bnpc}
We use the following easy lemma without referring to it.
\begin{lemma}\label{lemma=easy}
Let $X$ be a general uniquely geodesic space.
\begin{enumerate}[$(1)$]
  \item If $[x,y]\, ||\, [x,y']$, then $y=y'$.
  \item If $[x,y]\, || \, [x',y']$, then for each isometry $T$ on $X$, $[T\cdot x,T\cdot y]\, ||\, [T\cdot x',T\cdot y']$.
\end{enumerate}
\end{lemma}

One of the keys to the proof above in Subsection~\ref{subsection=banach} is that $\alpha(g)\cdot \xi -\xi \in E^{\rho(G)}$ for all $g\in Q$. To generalize it to non-linear case, recall our strategy of \textit{replacing ``$y-x$'' with the parallel equivalence class of $[x,y]$} from Subsection~\ref{subsection=heart}.

\begin{proposition}[\textit{Key Pseudo-Uniqueness of realizers}]\label{proposition=parallel2}
In the setting as in $(1)$ of Proposition~$\ref{proposition=key}$, for all $g\in Q$ and for all $\gamma\in G$,
\[
[x, \alpha(g)\cdot x]\, ||\,[\alpha(\gamma)\cdot x, \alpha(\gamma g)\cdot x]\quad and\quad [y, \alpha(g)\cdot y]\, ||\,[\alpha(\gamma)\cdot y, \alpha(\gamma g)\cdot y].\]
\end{proposition}

\begin{proof}
We only prove for $x$. First, we observe the following:
\begin{lemma}[Pseudo-Uniqueness of realizers]\label{lemma=parallel1}
In the setting as in $(1)$ of Proposition~$\ref{proposition=key}$, let $(x',y')$ is another realizer of $D$. Then, $[x,y]\, ||\, [x',y']$.
\end{lemma}
Indeed, by uniqueness of geodesics and isometry of $\alpha$, $m_1:=m_{1/2}[x,x'] \in X^{\alpha(M)}$ and $m_2:=m_{1/2}[y,y'] \in X^{\alpha(L)}$.  The minimality of $D$ and $(\mathrm{NPC}1)$ end the proof.

For the proof of Proposition~\ref{proposition=parallel2}, we argue in induction on the word length $n$ of $g$ with respect to $P\cup P^{-1}$. If $n=0$, then the assertion holds. 

We consider the case that $g \in P\cup P^{-1}$ ($n=1$), and change symbols from $g$ to $h$. We first claim that the parallelism in Proposition~\ref{proposition=parallel2} holds true if $\gamma \in M\cup L$. Indeed, this is trivial if $\gamma \in M$. If $\gamma \in L$, then $[y, \alpha(h)\cdot y]\, ||\,[\alpha(\gamma)\cdot y, \alpha(\gamma h)\cdot y]$. Observe that $[x, \alpha(h)\cdot x]\, ||\,[y, \alpha(h)\cdot y]$ and that $[\alpha(\gamma)x, \alpha(\gamma h)\cdot x]\, ||\,[\alpha(\gamma)y, \alpha(\gamma h)\cdot y]$. Employ $(\mathrm{TP})$, and obtain the conclusion. Secondly, for a general $\gamma \in G$, repeat a similar argument to one above by using $(\mathrm{TP})$ (recall that $\langle M,L\rangle =G$).

We now proceed in the induction step. For $g$ with $n\geq 2$, write as $g=hg_0$ where $h\in P\cup P^{-1}$ and $g_0$ has word length $n-1$. Let $\gamma \in G$. Then by the assertion for $n=1$, $[x, \alpha(h)\cdot x]\, ||\,[\alpha(\gamma)\cdot x, \alpha(\gamma h)\cdot x]$. By $(\mathrm{NPC}2)$, it implies that
\[
[x, \alpha(\gamma)\cdot x]\, ||\,[\alpha(h)\cdot x, \alpha(\gamma h)\cdot x].
\]
Moreover, by induction hypothesis, for all $\gamma'\in G$, $[x, \alpha(g_0)\cdot x]\, ||\,[\alpha(\gamma')\cdot x, \alpha(\gamma' g_0)\cdot x]$. By letting $\gamma'=h$ and $\gamma'=\gamma h$, we, respectively, obtain that $[x, \alpha(g_0)\cdot x]\, ||\,[\alpha(h)\cdot x, \alpha(g)\cdot x]$ and that $[x, \alpha(g_0)\cdot x]\, ||\,[\alpha(\gamma h)\cdot x, \alpha(\gamma g)\cdot x]$. By $(\mathrm{TP})$, these imply that $[\alpha(h)\cdot x, \alpha(g)\cdot x] \, ||\,[\alpha(\gamma h)\cdot x, \alpha(\gamma g)\cdot x]$. Again by $(\mathrm{NPC}2)$, 
\[
[\alpha(h)\cdot x,\alpha(\gamma h)\cdot x ] \, ||\,[\alpha(g)\cdot x, \alpha(\gamma g)\cdot x].
\]

By  $(\mathrm{TP})$, hence, $[x, \alpha(\gamma)\cdot x]\, ||\, [\alpha(g)\cdot x, \alpha(\gamma g)\cdot x]$. Finally, apply $(\mathrm{NPC}2)$.
\end{proof}

\begin{proof}[Proof of $(1)$ of Proposition~$\ref{proposition=key}$]
We only prove for $x$. Let $g_1,g_2 \in Q$. The goal is to show that $\Psi_x(g_1g_2)=\Psi_x(g_2g_1)$. By letting $g=g_2$ and $\gamma=g_1$ in Proposition~\ref{proposition=parallel2},  $[x, \alpha(g_2)\cdot x]\, ||\, [\alpha(g_1)\cdot x, \alpha(g_1 g_2)\cdot x]$. By $(\mathrm{NPC}2)$, $[x, \alpha(g_1)\cdot x]\, ||\, [\alpha(g_2)\cdot x, \alpha(g_1 g_2)\cdot x]$. In addition, by letting $g=g_1$ and $\gamma=g_2$ in Proposition~\ref{proposition=parallel2}, we obtain that $[x, \alpha(g_1)\cdot x]\, ||\, [\alpha(g_2)\cdot x, \alpha(g_2 g_1)\cdot x]$. By $(\mathrm{TP})$ and these two,  $[\alpha(g_2)\cdot x, \alpha(g_1 g_2)\cdot x] \, ||\, [\alpha(g_2)\cdot x, \alpha(g_2 g_1)\cdot x]$. Therefore, $\alpha(g_1 g_2)\cdot x=\alpha(g_2 g_1)\cdot x$, as desired.

Finally, we verify the last assertion in Proposition~\ref{proposition=key}. Assume that $Q^{\mathrm{abel}}$ is torsion. Let $g\in Q$. Then, there exists $N\in \mathbb{N}_{>0}$ such that $[g^N]=0$ in $(Q^{\mathrm{abel}},+)$, where $[\cdot]$ is the equivalence class of $Q\twoheadrightarrow Q^{\mathrm{abel}}$. Then, $\Psi_x(g^N)=x$. By Proposition~\ref{proposition=parallel2}, $[x,\alpha(g)\cdot x]\,||\, [\alpha(g) \cdot x,\alpha(g^2)\cdot x]\,||\, \cdots \,||\,  [\alpha(g^{N-1})\cdot x,x]$. By $(\mathrm{NPC}3)$,  $\alpha(g)\cdot x=x$.
\end{proof}

\subsection{The case of $\mathrm{CAT}(0)$ spaces}\label{subsection=cat0} 
Recall that $\mathrm{CAT}(0)$ spaces satisfy $(\mathrm{NPC}3)$. Suppose, on the contrary to $(2)$ of Proposition~$\ref{proposition=key}$, that there exists $h\in P$ such that $\alpha(h)\cdot x \ne x$. Then, by $(1)$, $h$ must be of infinite order. By Proposition~\ref{proposition=parallel2} and $(\mathrm{NPC}3)$, the union $l_h:=\bigcup_{n\in \mathbb{Z}}[\alpha(h^{n-1})\cdot x, \alpha(h^{n})\cdot x]$ is a bi-infinite geodesic line. 
\begin{proof}[Proof of $(2)$ of Proposition~$\ref{proposition=key}$]
We stick to the setting of the paragraph above. Then, Proposition~\ref{proposition=parallel2} implies that two geodesic lines $l_{h}$ and $\alpha(\gamma)\cdot l_{h}$ are asymptotic. 

We define a new binary relation ``$||_{\perp}$'': we write that $[z,w]\, ||_{\perp}\, [z',w']$ if $[z,w]\,|| \, [z',w']$ and if, besides, $d(z,z')$ equals $\mathrm{dist}([z,w],[z',w'])$. The proof of Lemma~\ref{lemma=parallel1}  shows that $[x,x']\, ||_{\perp}\, [y,y']$ for another realizer $(x',y')$ of $D$. We claim that for all $\gamma\in G$, $[x,\alpha(h)\cdot x]\, ||_{\perp}\, [\alpha(\gamma)\cdot x,\alpha(\gamma h)\cdot x]$. Indeed, first, consider the case $\gamma \in M\cup L$. Then, apply \cite[Lemma~2.15]{BH} for general $\gamma$ by using asymptotics of $\alpha(\gamma)\cdot l_{h}$'s.

By setting $\gamma=h$, we obtain that $[x, \alpha(h)\cdot x]\, ||_{\perp}\,[\alpha(h)\cdot x, \alpha(h^2)\cdot x]$. This contradicts that $\alpha(h)\cdot x\neq x$. Therefore, $x\in X^{\alpha(\langle P\rangle)}=X^{\alpha(Q)}$. Similarly, $y\in X^{\alpha(Q)}$.
\end{proof}

\section{Metric ultraproducts, scaling limits, and realizers of distances}\label{section=up}

The following well-known two propositions indicate \textit{how to find a realizer}. For completeness, we briefly recall proofs, which employ \textit{metric ultraproducts}.

\begin{definition}[Displacement and uniform action]
Let $S$ be a finite   generating set of $G$. 
Let $Z$ be a metric space and $\beta\colon G\curvearrowright Z$ be a action.
\begin{enumerate}[$(1)$]
  \item The \textit{displacement function with respect to $S$} is defined by
\[
\mathrm{disp}_{\beta,S}\colon Z\to \mathbb{R}_{\geq 0};\quad z\mapsto \max_{s\in S} d(z,\beta(s)\cdot z).
\]
   \item The action $\beta$ is said to be \textit{$1$-uniform with respect to $S$} if $\inf_{z\in Z}\mathrm{disp}_{\beta,S}(z)\geq 1$.
\end{enumerate}
\end{definition}

Note that $z\in Z^{\beta(G)}$ if and only if $\mathrm{disp}_{\beta,S}(z)=0$. Fix  $\mathcal{X}$. Then, we set the following three classes of actions (with spaces):
\begin{itemize}
  \item $\mathcal{C}_{\mathcal{X}}:=\{(\alpha,X):X\in \mathcal{X},\ \alpha\colon G\curvearrowright X\}$ (recall from  our convention that actions $\alpha$ are assumed to be isometric);
  \item $\mathcal{C}_{\mathcal{X}}^{\mathrm{non}\textrm{-}\mathrm{fixed}}:=\{(\alpha,X):(\alpha,X )\in \mathcal{C}_{\mathcal{X}},\ X^{\alpha(G)}=\emptyset\}$; and
  \item $\mathcal{C}_{\mathcal{X}}^{(S,1)\textrm{-}\mathrm{uniform}}:=\{(\alpha,X):(\alpha,X )\in \mathcal{C}_{\mathcal{X}},\ \textrm{$\alpha$ is $1$-uniform with respect to $S$.}\}$.
\end{itemize}
Note that $\mathcal{C}_{\mathcal{X}}^{(S,1)\textrm{-}\mathrm{uniform}}\subseteq \mathcal{C}_{\mathcal{X}}^{\mathrm{non}\textrm{-}\mathrm{fixed}}$. Property $(\mathrm{F}_{\mathcal{X}})$ exactly says that $\mathcal{C}_{\mathcal{X}}^{\mathrm{non}\textrm{-}\mathrm{fixed}}= \emptyset$.

\begin{proposition}[Gromov--Schoen argument]\label{proposition=gs}
Let $\mathcal{X}$ be a class of metric spaces that satisfies $(\mathrm{S}_{+})$ and $(\mathrm{U})$. Assume that $G$ fails to have property $(\mathrm{F}_{\mathcal{X}})$. Then for every $($equivalently, some$)$ finite   generating set $S$ of $G$, $\mathcal{C}_{\mathcal{X}}^{(S,1)\textrm{-}\mathrm{uniform}}\ne \emptyset$.
\end{proposition}

\begin{proposition}[III. 3--4 of \cite{shalom2006}]\label{proposition=shalom}
Let $\mathcal{X}$ be a class of metric spaces that satisfies $(\mathrm{U})$. Let $M\leqslant G$ and $L\leqslant G$ such that $\langle M,L\rangle =G$. Let $S$ be a finite   generating set of $G$. Assume that $G\geqslant M$ and $G\geqslant L$ have relative property $(\mathrm{F}_{\mathcal{X}})$; and that $\mathcal{C}_{\mathcal{X}}^{(S,1)\textrm{-}\mathrm{uniform}}\ne \emptyset$. Then, the following infimum is realized:
\[
D_S:=\inf \{d_X(x,y): (\alpha,X) \in \mathcal{C}_{\mathcal{X}}^{(S,1)\textrm{-}\mathrm{uniform}},\ x\in X^{\alpha(M)},\ y\in X^{\alpha(L)}\}.
\]
\end{proposition}

\subsection{(Pointed) metric ultraproducts}\label{subsection=up}
We briefly recall the definitions on (pointed) metric ultraproducts. We refer the reader to a survey \cite{stalder} for more details.

\textit{Ultrafilers} $\mathcal{U}$ on $\mathbb{N}$ have one-to-one correspondence to $\{0,1\}$-valued probability \textit{mean}s (that means, \textit{finitely additive} measures $\mu$ defined over all subsets in $\mathbb{N}$ with $\mu(\mathbb{N})=1$) in the following manner: $\mathcal{U}=\{A\subseteq \mathbb{N}: \mu(A)=1\}$. \textit{Non-principal} ultrafilters correspond to  such means that are \textit{not} Dirac masses.

In what follows, we fix a non-principal ultrafilter $\mathcal{U}$. For a real sequence $((r_n))_{n\in \mathbb{N}}$, we write as $\lim_{\mathcal{U}}r_n=r_{\infty}$, if for all $\epsilon >0$, $\{n\in \mathbb{N}: |r_{\infty}-r_n|<\epsilon\}\in \mathcal{U}$ holds. Then, it is well-known that  every \textit{bounded} real sequence $(r_n)_n$ has a (unique) limit with respect to $\mathcal{U}$. Also, if $\lim_{n\to \infty}r_{n}=r_{\infty}$, then $\lim_{\mathcal{U}}r_n=r_{\infty}$.

Let $((X_n,d_n,w_n))_{n\in \mathbb{N}}$ be a sequence of pointed metric spaces (that means, $w_n\in (X_n,d_n)$). Let $(\sum (X_n,w_n))_{\ell_{\infty}}:=\{(x_n)_n: \sup_{n\in \mathbb{N}} d(x_n,w_n)<\infty\}$. Define a semi-metric $d_{\infty}$ on it by $d_{\infty}((x_n)_n,(y_n)_n):=\lim_{\mathcal{U}}d_n(x_n,y_n)$. Define the \textit{pointed metric ultraproduct} $(X_{\mathcal{U}},d_{\mathcal{U}},w_{\mathcal{U}})$ by $X_{\mathcal{U}}:=(\sum (X_n,w_n))_{\ell_{\infty}}/\sim_{d_{\infty}=0}$. Here $d_{\mathcal{U}}$ is the induced metric, and $w_{\mathcal{U}}:=[(w_n)_n]$. We write the ultraproduct as $\lim_{\mathcal{U}}(X_n,d_n,w_n)$.

Now we fix $G$ and a finite   generating set $S$. For a sequence of pointed (isometric) $G$-actions $(\alpha_n,X_n,w_n)$, under the condition
\[
\sup_{n}\mathrm{disp}_{\alpha_n,S}(w_n) <\infty, \tag{$\diamond$}
\]
we can define the \textit{pointed metric ultraproduct action} $\alpha_{\mathcal{U}}$ on $(X_{\mathcal{U}},w_{\mathcal{U}})$ by $\alpha_{\mathcal{U}}(\gamma)\cdot [(x_n)_n]:=[(\alpha_n(\gamma)\cdot x_n)_n]$. We also write this ultraproduct action as $\lim_{\mathcal{U}}(\alpha_n,X_n,w_n)$.

\subsection{Sketch of proofs of two propositions.}
In \cite{mimuraT}, we describe the case that $\mathcal{X}=\mathcal{H}\mathrm{ilbert}$. See also \cite[Lemma~3.3 and Theorem~3.11]{stalder}.

\begin{proof}[Proof of Proposition~$\ref{proposition=gs}$]
Let $(\alpha,X)\in \mathcal{C}_{\mathcal{X}}^{\mathrm{non}\textrm{-}\mathrm{fixed}}(\ne \emptyset)$. If $\inf_{x\in X} \mathrm{disp}_{\alpha,S}(x)>0$, then by scaling up we obtain a $1$-uniform action. 

Otherwise, apply \cite[Lemma~3.3]{stalder}. Then, we obtain a sequence $(w_n)_n$ in $X$ with the following properties: $\sup_{n}\mathrm{disp}_{\alpha,S}(w_n)<1$; and for all $x\in X$ with $d(x,w_n)\leq (n+1)\mathrm{disp}_{\alpha,S}(w_n)$, $\mathrm{disp}_{\alpha,S}(x)\geq \mathrm{disp}_{\alpha,S}(w_n)/2$. The ultraproduct action $\lim_{\mathcal{U}}(\alpha, X, 2(\mathrm{disp}_{\alpha,S}(w_n))^{-1}d_X, w_n)$ 
works. Note that $2(\mathrm{disp}_{\alpha,S}(w_n))^{-1}>1$.
\end{proof}

\begin{proof}[Proof of Proposition~$\ref{proposition=shalom}$]
Note that $D_S$ is the infimum over a \textit{non-empty} set. Let $((\alpha_n,X_n,x_n,y_n))_n$ be a sequence such that $d_n(x_n,y_n)\leq D_S+2^{-n}$. Then, we claim that $((\alpha_n,X_n,x_n))_{n\in \mathbb{N}}$ satisfies $(\diamond)$. Indeed, $\sup_{\gamma\in M} d_n(x_n,\alpha_n(\gamma)\cdot x_n)=0$ and $\sup_{\gamma\in  L} d_n(y_n,\alpha_n(\gamma)\cdot y_n)=0$. By isometry of $\alpha_n$ and triangle inequality, $\sup_{\gamma\in M\cup L} d_n(x_n,\alpha_n(\gamma)\cdot x_n)\leq 2(D_S+2^{-n})$; and hence $\mathrm{disp}_{\alpha_n,S}(x_n) \leq 2N(D_S+2^{-n})\leq 2N (D_S+1)$, as claimed. Here, $N:=\max_{s\in S}|s|_{M\cup L}$, where $|\cdot|_{\cdot}$ is the word length.

The resulting action $(\alpha,X):=\lim_{\mathcal{U}}(\alpha_n,X_n)$, with base points $(x_n)_n$, is in $ \mathcal{C}_{\mathcal{X}}^{(S,1)\textrm{-}\mathrm{uniform}}$. This $(\alpha,X)$, and points $x:=[(x_n)_n]$ and $y:=[(y_n)_n]$, realize $D_S$.
\end{proof}

\section{Proofs of Theorem~\ref{mtheorem=main1} and Theorem~\ref{mtheorem=main1b}}\label{section=proof}

The key to extending our self-improvement argument (Proposition~\ref{proposition=key}) for type $(\mathrm{I\hspace{-.1em}I})$ moves is the $\mathrm{Aut}(G)$-action on $\mathcal{C}_{\mathcal{X}}$ by twisting: for $\phi \in \mathrm{Aut}(G)$, $(\alpha,X)\mapsto (\alpha^{\phi},X);\ \alpha^{\phi}(\gamma):=\alpha(\phi(\gamma))$. For a special choice of subgroups in $\mathrm{Aut}(G)$ and \textit{that of } $S$, \textit{the subgroup action keeps $\mathcal{C}_{\mathcal{X}}^{(S,1)\textrm{-}\mathrm{uniform}}$ invariant}. In Proposition~\ref{proposition=key}, we transform the points $x$ and $y$ inside the action. Instead, we, here, \textit{transform the action itself}.

\begin{proof}[Proof of Theorem~\ref{mtheorem=main1}]

Suppose, on the contrary, that $G$ fails to have property $(\mathrm{F}_{\mathcal{X}})$.

\noindent
\textit{Step~$0$.} By hypothesis $(ii)$, we can choose a \textit{finite}   generating set $S$ of $G$ such that \textit{$S$ is $\Pi$-invariant}. We fix such $S$ throughout the proof.

\noindent
\textit{Step~$1$.} By Propositions~\ref{proposition=gs} and \ref{proposition=shalom}, $D=D_S$ in Proposition~\ref{proposition=shalom} is realized. 

\noindent
\textit{Final step.} We claim the following. By $(\mathrm{GAME})$, it will lead us to a contradiction.

\begin{claim}[\textit{Full self-improvement argument}]\label{claim=finalclaim}
In the current setting, let $(\alpha,X,x,y)$ be $\mathrm{any}$ realizer of $D$. Namely, $(\alpha,X)$ is in $\mathcal{C}_{\mathcal{X}}^{(S,1)\textrm{-}\mathrm{uniform}}$; $x\in X^{\alpha(M)}$, and $y\in X^{\alpha(L)}$; and $d(x,y)=D$. Then, in each stage in $(\mathrm{Game})$, $x\in X^{\alpha(H_1)}$ and $y\in X^{\alpha(H_2)}$.
\end{claim}

\noindent
\textit{$($Proof of Claim~$\ref{claim=finalclaim}$$)$.}

We take an induction on  number of moves. In the initial stage, the assertion of Claim~$\ref{claim=finalclaim}$ holds. We separate our cases in terms of types of the move in the new stage. Here, we indicate by $H_1$ and $H_2$ the ones just before the newest move.

\begin{itemize}
  \item \textit{Case~$1$. Our new move is of type $(\mathrm{I})$:} treated by Proposition~\ref{proposition=key}. 
  \item \textit{Case~$2$. Our new move is of type $(\mathrm{I\hspace{-.1em}I})$:} we, here, only deal with the case of type $(\mathrm{I\hspace{-.1em}I}_{(12)})$ moves.  The crucial point  is that \textit{the $\mathrm{Inn}(G)\cdot \Pi$-action on $\mathcal{C}_{\mathcal{X}}$ leaves $\mathcal{C}_{\mathcal{X}}^{(S,1)\textrm{-}\mathrm{uniform}}$ invariant}. Indeed, \textit{for $\phi \in \Pi$, this follows from the choice of $S$ in Step~$0$}. For $\phi \in \mathrm{Inn}(G)$, observe that the twisted action of $\beta$ by an inner conjugation $\gamma \to \lambda \gamma \lambda^{-1}$ is intertwined to $\beta$ by $\beta(\lambda)^{-1}$ . 

Let $\phi\in \Lambda$. By induction hypothesis,  $y\in X^{\alpha(H_2)}$. Because $\phi(M)\leqslant H_2$ by the condition of $\Lambda$,  $y\in X^{\alpha^{\phi}(M)}$. Similarly, $x\in X^{\alpha^{\phi}(L)}$. As we saw that $(\alpha^{\phi},X) \in\mathcal{C}_{\mathcal{X}}^{(S,1)\textrm{-}\mathrm{uniform}} $, the quadruple $(\alpha^{\phi},X,y,x)$ is \textit{another realizer} of $D$. 

Recall that our induction hypothesis is not for a specific realizer, but for an \textit{arbitrary} realizer. Therefore, $y\in X^{\alpha^{\phi}(H_1)}=X^{\alpha(\phi(H_1))}$ and $x\in X^{\alpha^{\phi}(H_2)}=X^{\alpha(\phi(H_2))}$. (Note that our $H_1$ and $H_2$ are at the stage very before the newest move, and hence we may apply the induction hypothesis.) Thus, $H_1$ is enlarged to $\langle H_1,\bigcup_{\phi\in \Lambda}\phi(H_2)\rangle$; and $H_2$ is enlarged to $\langle H_2,\bigcup_{\phi\in \Lambda}\phi(H_1)\rangle$.
\end{itemize}
Our proof of Theorem~\ref{mtheorem=main1} is completed.
\end{proof}

\begin{proof}[Proof of Theorem~\ref{mtheorem=main1b}]

If $\mathcal{X}$ consists of $\mathrm{CAT}(0)$ spaces, then by $(2)$ of Proposition~\ref{proposition=key} ends our proof. In the case where $\mathcal{E}=\mathcal{X}$ consists of Banach spaces, by Remark~\ref{remark=bnpc}, every $E\in \mathcal{E}$ must be uniformly convex. We may assume that $\mathcal{E}$ contains a non-zero Banach space. Then, we claim that $G^{\mathrm{abel}}$ is finite. Indeed, by relative property $(\mathrm{F}_{\mathcal{E}})$, $\mathrm{ab}_G(M)$ and $\mathrm{ab}_G(L)$ are finite, where $\mathrm{ab}_G\colon G\twoheadrightarrow (G^{\mathrm{abel}},+)$ (note that finite generation of $G$ implies those of $\mathrm{ab}_G(M)$ and $\mathrm{ab}_G(L)$). Since $\mathrm{ab}_G(G)=\mathrm{ab}_G(M)+\mathrm{ab}_G(L)$, we are done. Finally, $(3)$ of Proposition~\ref{proposition=key} finishes our demonstration.
\end{proof}

\section{Further generalizations}\label{section=general}

\subsection{Superintrinsic synthesis in full generality}\label{subsection=generalproof}
Here, we prove Theorem~\ref{mtheorem=main2}. Let $X$ satisfy $(\Theta$-$\mathrm{NPC}1)$ $($for some fixed $\Theta$$)$. Let $M_1,\ldots ,M_l \leqslant G$ for $l\in \mathbb{N}_{\geq 3}$, and  $\alpha\colon G\curvearrowright X$ such that $X^{\alpha(M_i)}\ne \emptyset$ for all $i$. Then, for an $n$-ple $(z_1,\ldots ,z_l)$, where $z_i\in X^{\alpha(M_i)}$ for all $i$, we define the \textit{$\Theta$-energy} of it by $\mathbb{E}^{\Theta}(z_1,\ldots ,z_l):=\sum_{1\leq i<j\leq l}\Theta(d(z_i,z_j))^2$. We define the \textit{$\Theta$-energy of $(X^{\alpha(M_1)},\ldots ,X^{\alpha(M_l)})$}, written as $\mathbb{E}^{\Theta}(X^{\alpha(M_1)},\ldots,X^{\alpha(M_l)})$,  by the infimum of it over all such $(z_1,\ldots ,z_l)$.

\begin{lemma}[Pseudo-Uniqueness of realizers of $\Theta$-energy]\label{lemma=key2}
In the setting above, if  $(x_1,\ldots,x_l)$ and $(y_1,\ldots ,y_l)$ are two realizers of $D:=\mathbb{E}^{\Theta}(X^{\alpha(M_1)},\ldots,X^{\alpha(M_l)})$. Then, for all $1\leq i<j\leq l$, $[x_i,y_i]\, ||\, [x_j,y_j]$.
\end{lemma}

\begin{proof}
For all $i$, $m_i:=m_{1/2}[x_i,y_i]$ is in $X^{\alpha(M_i)}$. Then, by $(\Theta$-$\mathrm{NPC}1)$, 
\begin{align*}
\mathbb{E}^{\Theta}(m_1,\ldots ,m_l)&\leq \frac{1}{4}\sum_{1\leq i<j\leq n}\left\{\Theta(d(x_i,x_j))+\Theta(d(y_i,y_j))\right\}^2 \\
&\leq \frac{1}{2}\sum_{1\leq i<j\leq n}\left(\Theta(d(x_i,x_j))^2+\Theta(d(y_i,y_j))^2\right) =D.
\end{align*}
By minimality of $D$, all of the inequalities above, in fact,  must be equalities. 
\end{proof}

\begin{proof}[Proof of Theorem~\ref{mtheorem=main2}]
The proof is done by a simple modification of the proofs of Theorems~\ref{mtheorem=main1} and \ref{mtheorem=main1b}. More precisely, we focus on \textit{realizers of $\Theta$-energy} of fixed point subsets, and replace Lemma~\ref{lemma=parallel1} with Lemma~\ref{lemma=key2}. Details are left to the reader.
\end{proof}

\subsection{Generalization to supereflexive Banach spaces}\label{subsection=sr}

\begin{mtheorem}[Superintrinsic synthesis with respect to superreflexive Banach spaces]\label{mtheorem=sr}
Let $\mathcal{E}$ be a class of superreflexive Banach spaces. Assume that it satisfies $(\mathrm{U})$.
\begin{enumerate}[$(1)$]
  \item $($Two-subgroup case$)$ Let $M,L\leqslant G$ and $\Pi\leqslant \mathrm{Aut }(G)$. Assume hypotheses $(i)$, $(ii)$, and $(\mathrm{GAME}^+)$ as in Theorem~\ref{mtheorem=main1b}. Then, relative properties $(\mathrm{F}_{\mathcal{E}})$ for $G\geqslant M$ and $G\geqslant M$ imply property $(\mathrm{F}_{\mathcal{E}})$ for $G$.
  \item $($General case$)$ Let $l\in \mathbb{N}_{\geq 3}$. Let $M_i\leqslant G$ for $1\leq i \leq l$ and $\Pi\leqslant \mathrm{Aut }(G)$. Assume hypotheses $(i_l)$, $(ii)$, and $(\mathrm{GAME}_l^+)$ as in Theorem~\ref{mtheorem=main2}. Then, relative properties $(\mathrm{F}_{\mathcal{E}})$ for $G\geqslant M_i$ for all $1\leq i\leq l$ imply property $(\mathrm{F}_{\mathcal{E}})$ for $G$.
\end{enumerate}
\end{mtheorem}

\begin{proof}
Basic idea is to employ the following result \cite[Proposition~2.3]{BFGM}: for an (isometric) $G$-action $(\alpha, E)$ where $E$ is superreflexive, there exists an \textit{isometric} $G$-action $(\beta, F)$ and an isomorphism $T\colon E\stackrel{\simeq}{\to} F$ such that $F$ is uniformly convex and that $T$ intertwines $(\alpha,E)$ and $(\beta,F)$. Moreover, $T$ can be chosen to satisfy $\|T\|_{\mathrm{op}}\leq 2$ and $\|T^{-1}\|_{\mathrm{op}}\leq 1$. 

Suppose, on the contrary, that $G$ fails to have property $(\mathrm{F}_{\mathcal{E}})$. First, we take the same Step~0 as in the proof of Theorem~\ref{mtheorem=main1}, and have $S$. In Step~1, we consider a different class than $\mathcal{C}_{\mathcal{E}}^{(S,1)\textrm{-}\mathrm{uniform}}$ in the following way. By Proposition~\ref{proposition=gs}, there exists $(\alpha,E)\in \mathcal{C}_{\mathcal{E}}^{(S,1)\textrm{-}\mathrm{uniform}}$. Fix such $(\alpha,E)$. Then, apply the aforementioned result, and obtain $(\beta,F)$ and $T\colon E\stackrel{\simeq}{\to}F$ for this $(\alpha, E)$. Set $\overline{\mathcal{C}}$ be the class of all (isometric) actions $(\sigma',K')$ with the following properties:
\begin{itemize}
  \item the action $(\sigma',K')$ is $1$-uniform with respect to $S$;
  \item the Banach space $K'$ is uniformly convex with $\delta_{K'}(r)\geq \delta_{F}(r)$ for all $r\in (0,2]$ (recall the definition of $\delta$ from Definition~\ref{definition=superexpanders}); and
  \item there exist an (isometric) $G$-action $(\sigma,K) \in \mathcal{C}_{\mathcal{E}}$  and an isomorphism $V\colon K\stackrel{\simeq}{\to}K'$ such that $\|V\|_{\mathrm{op}} \leq 2$, $\|V^{-1}\|_{\mathrm{op}} \leq 1$, and $V$ intertwines $(\sigma,K)$ and $(\sigma',K')$.
\end{itemize}
Note that $\overline{\mathcal{C}}\ne \emptyset$ because $(\beta,F)\in \overline{\mathcal{C}}$. We replace $\mathcal{C}_{\mathcal{E}}^{(S,1)\textrm{-}\mathrm{uniform}}$ with $\overline{\mathcal{C}}$, and argue in similar ways to the proofs of Theorems~\ref{mtheorem=main1}, \ref{mtheorem=main1b} and \ref{mtheorem=main2}. 
\end{proof}

\section{Applications and Remarks}\label{section=application}
Here we prove Corollaries~\ref{corollary=st}, \ref{corollary=expanders}, and \ref{corollary=superexpanders}.

\begin{proof}[Proof of $(a)$ of Corollary~$\ref{corollary=st}$]
Combine Theorems~\ref{mtheorem=main1}, \ref{mtheorem=main1b}, \ref{mtheorem=sr}, and Example~\ref{example=st}. 
\end{proof}

What remains is the ``Part Step'' (relative property $(\mathrm{F}_{\mathcal{X}})$). The key is \textit{relative property $(\mathrm{T}_{\mathcal{X}})$}; and \textit{rank-raising argument}, observed in \cite{mimura2011} and \cite{LMS}. 

\subsection{Relative property $(\mathrm{T}_{\mathcal{E}})$, and rank-raising argument}\label{subsection=relT}
\begin{definition}[Relative property $(\mathrm{T}_{\mathcal{E}})$, \cite{BFGM} with modification by de la Salle]\label{definition=relbanach}
Let $\mathcal{E}$ be a class of Banach spaces. Let $\Lambda$ be a countable group and $M\trianglelefteq \Lambda$. We say that $\Lambda\trianglerighteq M$ has \textit{relative property $(\mathrm{T}_{\mathcal{E}})$} if the following holds true: for every $E\in \mathcal{E}$ and for every linear isometric representation $\rho \colon \Lambda \to O(E)$, there exists $\epsilon >0$ and a \textit{finite} subset $K\subseteq \Lambda$ such that for all $\zeta \in E$, $ \max_{s\in K} \|\rho(s)\zeta-\zeta \|   \geq  \epsilon  \|\overline{\zeta}\|_{\overline{E}}$. 
Here $\overline{E}=E/E^{\rho(M)}$, and $E \twoheadrightarrow \overline{E} \colon \zeta \mapsto \overline{\zeta}$ is the the quotient map.

For a finite subset $K\subseteq \Lambda$, we say that $\Lambda\trianglerighteq  M$ has \textit{relative property $(\mathrm{T}_{\mathcal{E}})$ with respect to $K$} if the condition above is satisfied for this specified $K$ $($for all $\rho$$)$.
\end{definition}
If $\Lambda$ is finitely generated, then relative property $(\mathrm{T}_{\mathcal{E}})$ is equivalent to that with respect to $S$, where $S$ is (some, equivalently, any) finite \textit{generating} set of $\Lambda$. Hence, \textit{if $\Lambda$ is finitely generated, then we usually do not specify $K$}. The open mapping theorem (for Fr\'{e}chet spaces) \cite[3.a]{BFGM} verifies that \textit{relative property $(\mathrm{F}_{\mathcal{E}})$  implies relative property $(\mathrm{T}_{\mathcal{E}})$}, for a pair $\Lambda\trianglerighteq M$. In a private communication, de la Salle pointed out that, to make the proof work, we may need to modify the original definition \cite[Definition~1.1]{BFGM} to Definition~\ref{definition=relbanach}. If $\mathcal{E}\subseteq \mathcal{B}_{\mathrm{uc}}$, then these two formulations are equivalent.

It is known that the converse implication to above is \textit{false} in general (for instance, set $\mathcal{E}=\mathcal{B}_{L_q}$ for sufficiently large $q$). Nevertheless, if we consider elementary/Steinberg groups, then \textit{by raising rank by $1$}, we obtain the following ``converse''.

\begin{definition}\label{definition=standard}
Let $R$ be of the form $\mathbb{F}_p\langle t_1,\ldots ,t_k\rangle$ or $\mathbb{Z}\langle t_1,\ldots ,t_k\rangle$, where $p$ prime and $k\in \mathbb{N}$. Let $m\geq 2$. Then, the \textit{standard finite subset of $\mathrm{E}(m,R)\ltimes R^m$} is defined as the the union of $\{e_{i,j}(\pm t_a):i\ne j \in [m], 0\leq a\leq k\}(\subseteq \mathrm{E}(m,R))$ and $\{(0,\ldots,0,\pm t_a ,0,\ldots ,0): 0\leq a\leq k\}(\subseteq R^{m})$. Here, $t_0$ means $1$, and for the latter set, the $\pm t_a$-term runs from the first to the $m$-th entries.

The \textit{standard finite subset of $\mathrm{St}(m,R)\ltimes R^m$} is defined by replacement of  $e_{i,j}(\pm t_a)$ with $E_{i,j}(\pm t_a)$ in the definition above.
\end{definition}
By $(\ast)$, the standard finite set is a \textit{generating} set, provided that $m\geq 3$.

\begin{proposition}[\textit{Rank-raising argument}, Theorem~1.3 in \cite{mimura2011} and Proposition~5.2 in \cite{LMS}]\label{proposition=rankraising}
Let $\mathcal{E}$ be a class of Banach spaces with respect to which fixed point property is equivalent to bounded orbit property. Let $R$ be of the form $\mathbb{F}_p\langle t_1,\ldots ,t_k\rangle$ or $\mathbb{Z}\langle t_1,\ldots ,t_k\rangle$, where $p$ prime and $k\in \mathbb{N}$. Let $m\geq 2$. Then, relative property $(\mathrm{T}_{\mathcal{E}})$  for $\mathrm{E}({m},R)\ltimes R^{m}\trianglerighteq R^{m}$, with respect to  the standard finite subset for $\mathrm{E}({m},R)\ltimes R^{m}$, implies relative property $(\mathrm{F}_{\mathcal{E}})$ for $\mathrm{E}({m+1},R)\ltimes R^{m+1}\geqslant R^{m+1}$. 

The corresponding assertion holds true if we replace $\mathrm{E}(m,R)$ and $\mathrm{E}(m+1,R)$, respectively, with $\mathrm{St}(m,R)$ and $\mathrm{St}(m+1,R)$. \end{proposition}

Note that if $m\geq 3$, then we do not need to specify such a finite subset. Recall from the paragraph below Theorem~\ref{theorem=shalom} that $\mathcal{E}$ fulfills the condition above if $\mathcal{E}\subseteq \{\mathrm{reflexive}\ \mathrm{Banach}\ \mathrm{space}\}\cup \mathcal{B}_{\mathrm{NC}L_1}$. The proof of Proposition~\ref{proposition=rankraising} goes along almost the same line as that of \cite[Proposition~5.2]{LMS}. One modification is to take ``Cartan-type involution'' ($  g\mapsto {}^t(g^{-1}) )$, we, besides, need  to apply a \textit{multiplication-reversed} ring isomorphism on $R$  to obtain a group isomorphism. By assumption on our $R$, the standard multiplication-reversed ring isomorphism (set $t_a \mapsto t_a$, and extend it as a multiplication-reversed ring homomorphism) does the job.

The following results are known on relative property $(\mathrm{T}_{\mathcal{E}})$.

\begin{theorem}\label{theorem=relTbanach}
Let $R$ be a ring as in Proposition~$\ref{proposition=rankraising}$. Let $m\geq 2$.
\begin{enumerate}[$(1)$]
  \item $($Kassabov, Genuine form of Theorem~$\ref{theorem=relT}$$)$ The pair $\mathrm{E}({m},R)\ltimes R^{m}\trianglerighteq R^{m}$ has relative property $(\mathrm{T}_{\mathcal{E}})$ with respect to the standard finite subset. Here $\mathcal{E}=\mathrm{H}\mathrm{ilbert}$.
  \item $($Olivier, \cite[Theorem 1.3]{olivier}; \cite[Theorem~A]{BFGM}; and \cite[Theorem 1.4]{mimura2011}$)$ The assertion above holds true also for $\mathcal{E}=\mathcal{B}_{\mathrm{NC}L_q}$ for all $1\leq q<\infty$; $\mathcal{E}=\mathcal{B}_{\mathrm{Q}L_q}$ for all $q\in (1,\infty)\setminus \{\frac{2j}{2j-1}:j\in \mathbb{Z}\}$; and for $\mathcal{E}=[\mathcal{H}\mathrm{ilbert}]$.
\end{enumerate}
Furthermore, all of these statements remain true if we replace $\mathrm{E}(m,R)$ with $\mathrm{St}(m,R)$.
\end{theorem}
\subsection{Closing of the proofs of Corollaries~\ref{corollary=st}, \ref{corollary=expanders}, and \ref{corollary=superexpanders}}\label{subsection=closing}

Since fixed point properties pass to group quotients, we may assume that $R$ is as in Proposition~$\ref{proposition=rankraising}$. 

\begin{corollary}[Our ``Part Step'']\label{corollary=relFbanach}
Let $\tilde{M},\tilde{L}\leqslant \mathrm{St}(n,R)=:\tilde{G}$ as in Corollary~$\ref{corollary=st}$. 
\begin{enumerate}[$(1)$]
  \item  For every $n\geq 3$, $\tilde{G}\geqslant \tilde{M}$ and $\tilde{G}\geqslant \tilde{L}$ have relative property $(\mathrm{F}_{\mathcal{H}\mathrm{ilbert}})$.
  \item For every $n\geq 4$, $\tilde{G}\geqslant \tilde{M}$ and $\tilde{G}\geqslant \tilde{L}$ has relative property $(\mathrm{F}_{\mathcal{E}})$, where $\mathcal{E}=\mathcal{B}_{\mathrm{NC}L_q}$ for all $1\leq q<\infty$; $\mathcal{E}=\mathcal{B}_{\mathrm{Q}L_q}$ for all $q\in (1,\infty)\setminus \{\frac{2j}{2j-1}:j\in \mathbb{Z}\}$; and $\mathcal{E}=[\mathcal{H}\mathrm{ilbert}]$.
\end{enumerate}
\end{corollary}

\begin{proof}
Combine Theorem~\ref{theorem=relTbanach} and Proposition~\ref{proposition=rankraising}. Then, we have the conclusions through a homomorphism $\mathrm{St}({n-1},R)\ltimes R^{n-1} \to \mathrm{St}(n,R)$ and a ``Cartan-type'' involution, as argued above. Here, $(1)$ follows for $n=3$ because for $\mathcal{E}=\mathcal{H}\mathrm{ilbert}$, relative property $(\mathrm{T}_{\mathcal{E}})$ is equivalent to relative property $(\mathrm{F}_{\mathcal{E}})$ (see Introduction).
\end{proof}

\begin{proof}[Proof of $(b)$ of Corollary~$\ref{corollary=st}$]
For the cases of $\mathcal{E}=\mathcal{B}_{\mathrm{NC}L_q}$ and $\mathcal{E}=\mathcal{B}_{\mathrm{Q}L_q}$, by $(1)$ of Example~\ref{example=npc}, $(a)$ of Corollary~\ref{corollary=st} ``Synthesizes'' Corollary~\ref{corollary=relFbanach}. For $[\mathcal{H}\mathrm{ilbert}]$, instead of it, consider $[\mathcal{H}\mathrm{ilbert}]_{C}$ for each $C\geq 1$. Recall that this satisfies $(\mathrm{U})$.
\end{proof}

\begin{proof}[Proof of Corollary~$\ref{corollary=expanders}$]
For $(1)$, observe that $\Lambda$ inherits property ``$(\mathrm{F}_{\mathcal{B}_{L_q}})$ for all $q\in (1,\infty)$'' from $\tilde{G}$ (by $\ell_q$-induction of isometric actions). Then, Theorem~\ref{theorem=relhyp} ends the proof. For $(2)$, apply $(2)$ and $(3)$ of Theorem~\ref{theorem=robust} (set $q=q$ for $\mathcal{B}_{\mathrm{NC}L_q}$ and $\mathcal{B}_{\mathrm{Q}L_q}$; and $q=2$ for $[\mathcal{H}\mathrm{ilbert}]_C$).
\end{proof}

\begin{proof}[Proof of Corollary~$\ref{corollary=superexpanders}$]
For each $E\in \mathcal{B}_{\mathrm{uc}}$, consider the class $\mathcal{T}_{\ell_2(\mathbb{N},E)}$ defined in Example~\ref{example=npc}.$(3)$. Combine Proposition~\ref{proposition=rankraising}, Corollary~\ref{corollary=st}.$(a)$,  and Lemma~\ref{lemma=superexpanders}.
\end{proof}

\subsection{Concluding remarks}\label{subsection=remarks}
\begin{remark}\label{remark=ngeq4}
We impose the assumption  ``rank$\geq 3$'' on Conjecture~\ref{conjecture=superexpanders} because of the rank-raising argument. At present, we do not know whether it is essential.
\end{remark}

\begin{remark}\label{remark=root}
Our synthesis may work for Steinberg groups $\mathrm{E}(\Phi, A)$ associated with a root system $\Phi$ over a unital, finitely generated and commutative ring $A$, if $\Phi$ is classical, reduced, irreducible, and of \textit{rank at least $3$} (see \cite{EJK} for definitions).  For instance, if $\Phi=C_n$ with $n\geq 3$, then the corresponding rank-raising argument was studied in the Ph.D. thesis \cite[Theorem 9.4.2]{mimuraphD} of the author.
\end{remark}

\begin{remark}\label{remark=noncommutativel1}
As we saw in the proof of Corollary~\ref{corollary=st}.$(b)$.$(2)$ for $q\in (1,\infty)$, ``Part Step'' remains true even for $q=1$. Nevertheless, in ``Synthesis Step'', we need to exclude the case of $q=1$ because non-commutative $L_1$-spaces are not superreflexive. We note that, on the other hand, Theorem~\ref{theorem=shalom} works even for $q=1$ due to \cite{BGM}. By Theorem~\ref{theorem=CK}, this provides the following result for $\mathrm{SL}(n,\mathbb{Z})$:

\begin{proposition}\label{proposiiton=slnz}
Let $n\geq 4$. Then, $\mathrm{SL}(n,\mathbb{Z})$ has property $(\mathrm{F}_{\mathcal{B}_{\mathrm{NC}L_1}})$.
\end{proposition}
In fact, it has property $(\mathrm{FF}_{\mathcal{B}_{\mathrm{NC}L_1}})$ in the sense of  \cite{mimura2011}. Indeed, Proposition~\ref{proposition=rankraising} is generalized to relative property $(\mathrm{FF}_{\mathcal{E}})$, see \cite[Proposition~5.2]{LMS}. Moreover, by the result of Nica \cite{nica}, the same holds for $\mathrm{SL}(n\geq 4,\mathbb{F}_p[t])$ for prime power $p$.

Study on property $(\mathrm{F}_{\mathcal{B}_{\mathrm{NC}L_1}})$ is of its own interest. This is because neither of approaches of Lafforgue \cite{lafforgue} and Oppenheim \cite{oppenheim} is applicable.
\end{remark}

Karlsson addressed the following to the author. He observed that $U\subseteq G$ Boundedly Generates $G$ if and only if the following holds; ``\textit{for every isometric action $G$ on a metric space, if $U$-orbits are bounded, then so are $G$-orbits}'' (to show ``if'' direction, consider $G\curvearrowright \mathrm{Cay}(G,U)$). From this viewpoint, (certain cases of) our superintrinsic synthesis may be regarded as follows: under certain superintrinsic hypotheses on $M_i$'s and $G$, bounded orbit property for the whole $G$ may be implied by the relative versions for $G\geqslant M_i$, \textit{for restricted classes of metric spaces}.

\begin{problem}[Karlsson]\label{problem=karlsson}
Establish  superintrinsic synthesis in $\mathrm{boundedness}$ $\mathrm{orbit}$ $\mathrm{properties}$ for classes of metric spaces, such as $($certain$)$ Gromov-hyperbolic spaces. 
\end{problem}

\section*{acknowledgments}
The author acknowledges Narutaka Ozawa for his comments and enduring encouragement. This paper is dedicated with admiration to him. The author is indebted to Masahiko Kanai for the terminology ``intrinsic''; Takayuki Okuda for discussions on the current formulation of moves in $(\mathrm{Game})$; and Yoshimichi Ueda for the suggestion of stating our criterion in terms of enlargement process (we use the terminology $(\mathrm{Game})$) during the author's visit at Kyushu University. He thanks \'{E}tienne Blanchard, Yves de Cornulier, Yuhei Suzuki, and Takumi Yokota for the careful reading of previous versions of the present manuscript and for various comments. The author  is grateful to Pierre de la Harpe for drawing the author's attention to a similarity to Mautner phenomena; Anders Karlsson for stimulating discussions; Takefumi Kondo, Hiroyasu Izeki, Shin Nayatani, and Tetsu Toyoda for discussions on parallelism in BNPC spaces; Bogdan Nica for his work \cite{nica}; Shin-ichi Oguni for providing references \cite{gerasimov} and \cite{puls}; and Mikael de la Salle for informing the author of work in progress. The author acknowledges Nicolas Monod and his secretary Marcia Gouffon for their hospitality on the author's current long-term stay in EPF Lausanne, supported by Postdoctoral Fellowship for Research Abroad no.~30-291 by the JSPS. The author is supported in part by  JSPS KAKENHI Grant Number JP25800033.

\bibliographystyle{amsalpha}
\bibliography{p.bib}

\end{document}